\documentclass[11pt,a4paper,leqno]{amsart}
\usepackage{amssymb,amsmath,amsthm}

\usepackage[mathscr]{eucal} % mathcal style
\usepackage{graphicx}
\usepackage{hyperref}
\usepackage{color}
\usepackage{pdfsync}
\usepackage{enumerate} % Fancy enumeration lists 
\usepackage{dsfont}
\usepackage{esint}

\usepackage{relsize} 
\usepackage{bigints}
\allowdisplaybreaks

\theoremstyle{plain}
\newtheorem{theorem}{Theorem}[section]
\newtheorem{definition}[theorem]{Definition}
\newtheorem{assumption}[theorem]{Assumption}
\newtheorem{lemma}[theorem]{Lemma}
\newtheorem{corollary}[theorem]{Corollary}
\newtheorem{proposition}[theorem]{Proposition}
\theoremstyle{remark}
\newtheorem{remark}[theorem]{Remark}

\newtheorem{example}[theorem]{Example}

\numberwithin{equation}{section}

\def\R{{\mathbb R}}% real numbers
\def\N{{\mathbb N}}% nonnegative integers
 \def\dd{\mathrm d}
\newcommand{\loc}{\rm loc}
\DeclareMathOperator{\diver}{div}
\DeclareMathOperator{\curl}{curl}
\newcommand{\dist}{\operatorname{dist}}
\newcommand{\supp}{\operatorname{supp}}
\renewcommand{\Cap}{\operatorname{Cap}}
\renewcommand{\leq}{\leqslant}\renewcommand{\le}{\leqslant}
\renewcommand{\geq}{\geqslant}\renewcommand{\ge}{\geqslant}

\def\cal#1{\mathcal{#1}}
\def\mb#1{\boldsymbol{#1}}

\def\eps{\varepsilon}

\newcommand{\beps}{b_{\eps}}

\newcommand{\veps}{v_{\eps}}

\newcommand{\weps}{\omega_{\eps}}

\newcommand{\Weps}{\Omega_{\eps}}
\newcommand{\Tweps}{\widetilde{\Omega}_{\eps}}
\newcommand{\phie}{\phi_{\eps}^1}
\newcommand{\alphae}{\alpha_{\eps}}
\newcommand{\Ieps}{\mathcal{I}_{\eps}}

\newcommand{\pun}{\psi_{\eps}^1}
\newcommand{\tphie}{\tilde{\phi}_\eps^1}
\newcommand{\Oo}{\Omega\setminus\{0\}}
\newcommand{\chid}{\chi_\delta}
\newcommand{\vreg}{v_{\text{reg}}}
\newcommand{\vsing}{v_{\text{sing}}}

\date\today

\title{Lake equations with an evanescent or emergent island}

\author[L.E. Hientzsch]{Lars Eric Hientzsch}

\author[C. Lacave]{Christophe Lacave}

\author[E. Miot]{Evelyne Miot}

\email{\newline larseric.hientzsch@univ-grenoble-alpes.fr \newline christophe.lacave@univ-grenoble-alpes.fr \newline evelyne.miot@univ-grenoble-alpes.fr}

\address{Univ. Grenoble Alpes, CNRS, Institut Fourier, F-38000 Grenoble, France.}

\begin{document}
\begin{abstract} We study the asymptotic dynamics of the lake equations in the following two cases, an island shrinking to a point and an emerging island. For both cases, we derive an asymptotic lake-type equation. In the former case, the asymptotic dynamics includes an additional Dirac mass in the vorticity. The main mathematical difficulty is that the equations are singular when the water depth vanishes. We provide new uniform estimates in weighted spaces for the related stream functions which will imply the compactness result. 
\end{abstract}

\maketitle

%\tableofcontents

\section{Introduction}
The purpose of this paper is to derive the asymptotic lake equations for the singular limit of an evanescent or emergent island. The lake equations arise as a $2D$ model for the vertically averaged horizontal component of the velocity of a $3D$ incompressible fluid \cite{Gr}. Given a domain $\Omega\subset \R^2$ and a topography (depth function) $b:\overline{\Omega}\rightarrow \R_+$, the pair $(\Omega,b)$ is called a lake and the lake equations read
\begin{equation}\label{eq:lake-vel}
 \begin{cases}
 \partial_t(bv)+\diver(bv\otimes v)+b\nabla p=0,\\
 \diver(bv)=0, \quad (bv)\cdot \mb n=0,
 \end{cases}
\end{equation}
where $v:\Omega\rightarrow \R^2$ is the velocity field, $p$ the pressure and $\mb n$ the outward-pointing unit normal vector on $\partial\Omega$. System \eqref{eq:lake-vel} can also be obtained from the shallow water wave equations in the low Froude number limit \cite{BM10}. For flat topographies, i.e. $b$ is a constant function, the lake equations \eqref{eq:lake-vel} simply become the incompressible two-dimensional Euler equations, for which the well-posedness has extensively been studied, see Wolibner \cite{Wolibner} or Yudovich \cite{yudo} for initial vorticity in $L^{\infty}(\Omega)$. When the depth function $b$ varies but is bounded away from zero, this analysis may be adapted to establish well-posedness of the lake equations \cite{LOTiti2}. More difficult and much more realistic, Bresch and M\'etivier \cite{BM} obtained the well-posedness on a simply-connected and smooth domain $\Omega$ for vorticity in $L^{\infty}(\Omega)$ with varying depth $b$ possibly vanishing on the boundary of $\Omega$. This required the use of elliptic techniques for degenerated equations. Therefore, we refer to \eqref{eq:lake-vel} on a lake $(\Omega,b)$ with vanishing topography as degenerated lake equations.
Subsequently, the second author with Nguyen and Pausader proved in \cite{LNP} that the lake equations are structurally stable under Hausdorff approximations of the fluid domain $\Omega$ and $L^p$ perturbations of the depth $b$. As a byproduct, the authors obtained the existence of a global weak solution in a large class of irregular lakes including not simply-connected domains, namely lakes with islands, as is recalled in Theorem~\ref{thm:existencenonsmooth} below. Existence of renormalized solutions for the degenerated lake equations has recently been introduced in \cite{BJ}.

In this paper, we are concerned firstly with the problem of an evanescent island, namely we aim to describe the limit flow for an island that coalesces to a point. Our study is motivated by scenarios such as flooding, sedimentation or erosion of an island due to e.g. extreme meteorologic events. The second issue we shall address is the adjoint problem where an island appears from the change of depth induced by geological events such as the birth of a volcano close to Mayotte \cite{Mayotte, Mayotte2}.
This example leads to a sequence of varying topographies, but our analysis also includes the scenario where the level of the water is uniformly decreasing. Let us mention that such asymptotic analysis are not included in \cite{LNP} where every asymptotic island is connected with at least two points, with the depth functions only vanishing on the boundary.
Here, we consider an asymptotic topography for \eqref{eq:lake-vel} such that $\Omega$ is a punctured lake, namely simply-connected and there exists a unique point $P\in\Omega$ with $b(P)=0$ (see Definition~\ref{ass:b}).

\subsection{Definition and sequences of lakes}

\subsubsection*{Some definitions of lakes}
We give in this section precise definitions of the lakes that will be under consideration in the main results. An island is said to be degenerated if it is reduced to a single point.
We start with the definition of a lake without degenerated island but with possibly one non-degenerated island. 
\begin{definition}[Lake without degenerated islands, with at most one non-degenerated island] \label{def-generallake}The lake $(\Omega,b)$ is a lake without degenerated islands and with one non-degenerated island if
\begin{enumerate}
\item $ \displaystyle \Omega := \widetilde{\Omega} \setminus \cal{I}$,
where $\widetilde \Omega$ is an open bounded simply connected subset of $\R^2$ and $\cal{I}$ is a compact, simply connected subset of $\widetilde \Omega$ containing at least two points;
\item $b\in L^\infty(\Omega, \R^+)$ and for any compact set $K\subset \Omega$, there exists a positive number $\theta_K$ such that $ b(x) \ge \theta_K>0$ on $K$;
\item there are small neighborhoods $\mathcal{O}^0$ and $\mathcal{O}^1$ of $\partial\widetilde{\Omega}$ and $\partial\cal{I}$ respectively, such that, for $k\in \{0,1\}$,
\begin{equation*}
b(x)=c(x)\left[d(x)\right]^{a_k} \qquad \text{ in } \mathcal{O}^k\cap \Omega,
\end{equation*}
where $c(x), d(x)$ are bounded functions in the neighborhood of the boundary, $c(x)\ge \theta>0$, $a_k \ge 0 $. Here the geometric function $d(x)$ is $C^1$ and satisfies $\Omega=\{d>0\}$ and $\nabla d\ne0$ on $\partial\Omega$;
\item if $b$ vanishes on the boundaries $\partial\widetilde \Omega$ and $\partial\cal{I}$, i.e. $a_k>0$, then the respective boundary is a $C^1$ Jordan curve.
\end{enumerate}

The lake $(\Omega,b)$ is a lake without degenerated islands and without non-degenerated island if
 $ \displaystyle \Omega = \widetilde{\Omega}$ is an open bounded simply connected subset of $\R^2$ and $b$ satisfies the same assumptions as above on $\Omega$.
\end{definition}

The second assumption states that the depth function may not vanish in the interior of the lake, and the third one states that the shore is either of non-vanishing ($a_k=0$) or vanishing topography with constant slopes $a_k>0$. We notice that up to a change of $c$, $\theta$, we may take $d(x)=\dist(x,\partial\Omega)$.

The first goal of this paper is to treat the asymptotic regime where a non-degenerated island $\mathcal{I}_{\eps}$ shrinks to a point $P\in \Omega$ as $\eps$ tends to zero. More precisely, the asymptotic lake has no non-degenerated island but $b(P)=0$ meaning that {\em (2)} is no longer satisfied. We specify now the assumptions on such lakes, referred to as \emph{punctured lakes}. Without loss of generality we assume $P$ to be $0$. Let $\Omega$ be a Jordan domain containing the origin, and let $b$ be a non-negative function on $\overline{\Omega}$ such that 
\[
b^{-1}(\{0\})= \partial \Omega \cup \{0\} \text{ or }b^{-1}(\{0\})= \{0\}.
\]
\begin{definition}[Punctured lakes]\label{ass:b} The lake $(\Omega,b)$ is a punctured lake (in $0$) if
\begin{enumerate}
\item $ \displaystyle \Omega := \widetilde{\Omega}$ is an open bounded simply connected subset of $\R^2$;
\item $b\in L^\infty(\Omega, \R^+)$ and for any compact set $K\subset \Omega\setminus\{0\}$, there exists a positive number $\theta_K$ such that $ b(x) \ge \theta_K>0$ on $K$;
\item there are small neighborhoods $\mathcal{O}^0$ and $\mathcal{O}^1$ of $\partial\widetilde{\Omega}$ and $\{0\}$ respectively, such that, for $k\in \{0,1\}$;
\begin{equation*}
b(x)=c(x)\left[d(x)\right]^{a_k} \qquad \text{ in } \mathcal{O}^k\cap \Omega,
\end{equation*}
where $c(x)$ is a bounded function such that $c(x)\ge \theta>0$ for all $x\in \Omega$, $a_0\geq 0$ and $a_1\in(0,2)$. As for the lakes without degenerated island, we set $d(x)=\dist(x,\partial\Omega\cup \{0\})$;
\item if $b$ vanishes on $\partial\Omega$, then the boundary $\partial \Omega$ is a $C^1$ Jordan curve.
\end{enumerate}
\end{definition}
As already mentioned, the analysis of \cite{LNP} does not include the case of punctured lakes. Indeed, it is required in \cite{LNP} that the $H^1$ capacities of the islands are positive (which is equivalent to assuming that the connected compact subset $\cal I$ is not reduced to points \cite[Prop. 2.2]{GV-Lacave2}).

We point out that assuming a non-vanishing topography on the island, namely assuming $a_{1}=0$, is not very natural for the applications we have in mind and would rather be an extension of the analysis developed in \cite{ILL1} for the 2D Euler equations. The degeneracy $b(0)=0$ in the limit topography is essential for our method, see the uniform estimates in Section~\ref{sec:BS} only available for $a_1>0$.

As showcase topography for a punctured lake is $b(x)=|x|^{a}$ in the neighborhood of $0$. For $a\in (0,1)$, the depth function $|x|^{a}$ defines a steep beach. In particular, $\nabla|x|^{a}$ is unbounded for $x\rightarrow 0$. For $a>1$, we notice that the depth corresponds to a flat beach; $\nabla|x|^{a}$ is bounded and tends to $0$ as $x\rightarrow 0$. We shall assume that $a\in (0,2)$ that will ensure that the velocity field $v$ obtained in the limit is at least an $L_{\loc}^1$-function, which will be crucial to identify $\curl v$, see Section~\ref{sec:evanescent}. Beyond the radial profiles $|x|^{a}$, more general and realistic depth functions are included in our assumptions.

\begin{remark}\label{rem:b}
We summarize here some properties of the depth functions $b$ that will be useful in the sequel.
We notice that (\emph{3}) in Definition~\ref{ass:b} ensures ${\sqrt{b}^{-1}}\in L_{\loc}^q(\Omega)$ for all $q<\frac{4}{a_{1}}$. In particular, $\sqrt{b}^{-1}\in L_{\loc}^{2+}(\Omega)$, which will be important to identify the PDE verified in $\Omega$ (see our first main result: Theorem~\ref{thm:maingeneral}). Because of the outer boundary, we can state that there is some $q\geq 2$ such that $\sqrt{b}^{-1}\in L^q(\Omega)$ only if $a_0<1$. However, local integrability is sufficient for our purpose and we allow any $a_{0}\geq 0$. 
\end{remark}

\subsubsection*{The evanescent island}
We introduce now the small island problem. Let $(\Omega_{\varepsilon},b_{\varepsilon})$ be a lake with one non-degenerated island $\cal I_{\varepsilon}$, which shrinks to one point. The main application that we have in mind is the flooding of an island, where the level of water is increasing:
\begin{equation*}
b_{\varepsilon}:= b-\varepsilon, \quad \Omega_{\varepsilon}:= \{x\in \overline{\Omega}, \ b_{\varepsilon}>0\}. 
\end{equation*}
Natural conditions on $b$ would imply that $(\Omega_{\varepsilon},b_{\varepsilon})$ is a lake with one island (in the sense of Definition~\ref{def-generallake} with $a_{0,\varepsilon}=a_{1,\varepsilon}=1$):
\[
\Omega_{\varepsilon} = \widetilde{\Omega}_{\varepsilon}\setminus\cal I_{\varepsilon}, \qquad \Ieps=\{b\leq \eps\}\cap B(0,1).
\]
For this example, we observe that $\widetilde{\Omega}_{\varepsilon}$ describes an increasing sequence converging to $\Omega$ in the Hausdorff sense. 

\begin{remark}\label{rem:gamma}
We refer to \cite[App. B]{GV-Lacave1} for a short introduction on the Hausdorff topology. Throughout this paper, we will not use the general deep properties of the Hausdorff convergence. We will only use that for any compact set $K\subset \Omega$, there exists $\varepsilon_{K} > 0$ such that $K \subset \Omega_{\varepsilon}$ for all $\varepsilon \in (0, \varepsilon_K)$. For the outer boundary $\partial\widetilde{\Omega}_{\varepsilon}$, we will also use in Section~\ref{sec:emergent} that Hausdorff convergence implies $\gamma$-convergence and the related Mosco's convergence (see \cite[App. C]{GV-Lacave1} or \cite[App. B]{LNP} for a brief overview of these notions). 
\end{remark}

\begin{remark}\label{rem:time}
In the present paper, the changes of the domain $\Omega$ and topography $b$ leading to the evanescent island are modeled by the dependence of the lake on the parameter $\eps$ rather than a time-dependence. This choice is motivated by the fact that geometry and topography changes of the lake occur at much larger time scales compared to the relevant time-scales for \eqref{eq:lake-vel}. Allowing for a time dependent domain and topography remains a complicated and interesting open problem that we leave for future work. Our main results provide a first step towards such a mathematical analysis. 
\end{remark}

Another situation that we want to include in our analysis is when the island disappears by erosion, where the topography is deformed by wind, water, or other natural agents (with possible constant volume of water).

To encompass these two cases, we only assume the following.
\begin{assumption}\label{ass:evanescent} Let $(\Omega_{\varepsilon},b_{\varepsilon})$ be a sequence of lakes with one non-degenerated island (as in Definition~\ref{def-generallake}) and $(\Omega,b)$ a punctured lake (as in Definition~\ref{ass:b}), such that 
\begin{enumerate}
\item ${\Omega}_{\varepsilon}\to \Omega\backslash\{0\}$ in the Hausdorff sense and $\Weps\subset \Omega$ for all $\eps$;
\item there exists $C>0$ such that $\beps(x)\leq C b(x)$ for all $x\in \Omega$ upon extending $\beps$ by $0$ to $\Omega$;
\item $\beps \rightarrow b$ in $L_{\loc}^1(\Oo)$ and for any $K\Subset \Oo$, there exists $\theta_{K}>0$ and $\varepsilon_{K}>0$ such that $\beps(x)\geq \theta_{K}$ for all $(x,\varepsilon)\in K\times(0,\varepsilon_{K}]$.
\end{enumerate}
\end{assumption}

Here, we require a quite general information concerning the way how the bottom tends to a degenerated bottom, provided by (\emph{1})-(\emph{2}) Assumption~\ref{ass:evanescent}.
These assumptions are suitable for the applications given above. Actually, we will discuss in Section~\ref{sec:BS} even more general assumptions, but more technical to state.

As already mentioned, if $\beps$ is a constant function (in $x$), \eqref{eq:lake-vel} reduces to the incompressible $2D$-Euler equations, see \cite{ILL1,L07} for the corresponding small obstacle problem.
Let us note that in these two papers, the authors required specific information on the shrinking process: the obstacle shrinks homothetically $\cal I_{\varepsilon}=\varepsilon \cal I_{1}$.

\subsubsection*{The emergent island}

Another natural question is to study the situation where the level of water is decreasing, starting from a case without island namely such that $b_\eps>0$ on $\Omega_{\varepsilon}$, $b_{\varepsilon}(0)=\varepsilon>0$ but at the limit $b(0)=0$. 

Therefore, we still consider an asymptotic punctured lake $(\Omega,b)$ verifying Definition~\ref{ass:b} and the typical situation we want to analyze is $b_{\varepsilon}:= b+\varepsilon$, $\varepsilon>0$.
From a modelling point of view, this scenario is mostly meaningful if $a_{0}\leq 1$. Indeed, for $a_0\in(1,2)$ the beach and the surface line of the water form a cusp. Thus, there seems to be no reasonable choice of $(\Weps,\beps)$ leading to the desired limit topography. For steep beaches, namely $a_0\in(0,1)$, it is natural to consider $\Weps=\Omega$ and to notice that $(\Omega_{\varepsilon},b_{\varepsilon})$ is a lake without island with non-vanishing topography : $a_{0,\varepsilon}=0$. For $a_0=1$, one may consider a sequence of lakes $(\Omega_{\varepsilon},b_{\varepsilon})$ without island and with $a_{0,\varepsilon}=1$.

Another situation that we want to include in our analysis is when the bottom is deformed (with possible constant level of water). One may think of an island born from a submarine volcano \cite{Mayotte,Mayotte2}. A showcase topography compatible with this configuration is 
 \begin{equation*}
 \beps(x)= (|x|^2+\eps )^{\frac{a_{1}}2}\eta(x) +b(x) (1-\eta(x))
\end{equation*}
where $\eta$ a cutoff function such that $\eta\equiv 1$ on $B(0,\varepsilon_{0})$ and $\eta\equiv 0$ on $\Omega\setminus B(0,2\varepsilon_{0})$. As for the evanescent island, we assume the sequence of lakes $(\Weps,\beps)$ to be $\eps$-dependent but independent of time, see Remark~\ref{rem:time}.
The former examples are included in the following general assumptions.
\begin{assumption}\label{ass:emergent} 
 Let $(\Omega_{\varepsilon},b_{\varepsilon})$ be a sequence of lakes without island (as in Definition~\ref{def-generallake}) and $(\Omega,b)$ a punctured lake (as in Definition~\ref{ass:b}), such that 
\begin{enumerate}
\item ${\Omega}_{\varepsilon}\to \Omega$ in the Hausdorff sense;
\item $\beps\in L^{\infty}(\Weps)$ uniformly bounded and $ \sqrt{\beps}^{-1}\rightarrow\sqrt{b}^{-1}$ in $L_{\loc}^q(\Omega)$ for some $q>2$.
\end{enumerate}
\end{assumption}

Even if Assumption~\ref{ass:emergent} seems weaker than Assumption~\ref{ass:evanescent}, we should be aware that (\emph{2}) encodes the emergence process.

\subsection{Main results}

As for the 2D-Euler equations, the notion of vorticity plays a prevalent role. Here, we introduce the potential vorticity $\omega$ as
\begin{equation*}
 \omega:=\frac{1}{b}\curl(v)=\frac{\partial_{1} v_{2}-\partial_{2} v_{1}}{b}
\end{equation*}
which satisfies the continuity equation together with the incompressibility condition,
\begin{equation}\label{eq:continuity}
 \partial_t(b\omega)+\diver(b v \omega)=0, \qquad \diver(bv)=0,
\end{equation}
which formally amounts to the transport equation
\begin{equation}\label{eq:transport}
 b\left(\partial_t\omega+v\cdot\nabla\omega\right)=0, \qquad \diver(bv)=0.
\end{equation}
 Because $b$ may vanish, equation \eqref{eq:transport} is a nonlinear advection equation with a degenerated anelastic constraint. We refer to the paper \cite{BJ} where stability estimates are derived for a class of such equations. In \cite{LNP}, it is crucial that $b$ only vanishes on the boundary, which implies that the div-curl problem is uniformly elliptic in the interior of the domain.

We notice that in the particular case $\Omega=(0,\infty)\times \R$ and $b(r,z)=r$, system \eqref{eq:transport} formally corresponds to the axis-symmetric Euler equations, where the quantity $\frac{\curl v}r$ also plays an important role, see \cite{GS}.

\medskip

Borrowed from \cite{LNP}, we give now the notion of weak solution to the vorticity formulation \eqref{eq:continuity} of \eqref{eq:lake-vel} that will be used for a given sequence of lakes and for asymptotic lakes.
\begin{definition}[Vorticity formulation]\label{defi:vorticity}
 Let $(\Omega, b)$ be a lake in the sense of Definitions~\ref{def-generallake} or \ref{ass:b}. Let $(v^0,\omega^0)$ be a pair such that
 \begin{equation*}
 \diver(b v^0)=0 \text{ in } \Omega \text{ and } b v^0\cdot \mb n=0 \text{ on } \partial\Omega \text{ in weak sense (see \eqref{imperm})}
 \end{equation*} 
 and
 \begin{equation*}
 \omega^0\in L^{\infty}(\Omega), \quad \curl v^0=b \omega^0 \quad \text{in the sense of distributions on $\Omega$}.
 \end{equation*}
The pair $(v,\omega)$ is a global weak solution of the vorticity formulation of \eqref{eq:lake-vel} on the lake $(\Omega,b)$ with initial condition $(v^0,\omega^0)$ if
 \begin{enumerate}[(i)]
 \item $\omega\in L^{\infty}(\R_+\times \Omega)$, $\sqrt{b}v\in L^{\infty}(\R_+;L^2(\Omega))$;
 \item $\diver(bv)=0$ in $\Omega$ and $bv\cdot\mb n=0$ on $\partial\Omega$ in weak sense, see \eqref{imperm2};
 \item $\curl v=b\omega$ in the sense of distributions on $\R_+\times \Omega$;
 \item for all $\Phi\in C_c^{\infty}([0,\infty)\times \Omega)$ it holds
 \begin{equation}\label{eq:vorticity}
 \int_0^{\infty}\int_{\Omega}\partial_t \Phi b\omega\dd x\dd t+\int_0^{\infty}\int_{\Omega}\nabla\Phi\cdot bv\omega\dd x \dd t+\int_{\Omega}b\omega^0\Phi(0)\dd x=0.
 \end{equation}
 \end{enumerate}
\end{definition}
The divergence and tangency boundary conditions are verified in a weak sense that will be defined in the following section (namely, in \eqref{imperm} and in \eqref{imperm2}).

We recall next the result obtained by the second author together with Pausader and Nguyen in \cite[Theorem 1.6]{LNP}: 
\begin{theorem}[\cite{LNP}]\label{thm:existencenonsmooth}
 Let $(\Omega,b)$ be a lake satisfying Definition~\ref{def-generallake}. For any $\omega^0\in L^\infty(\Omega)$ and $\gamma\in \R$, there exists a global weak solution $(v,\omega)$ of the lake equations in the vorticity formulation on the lake $(\Omega,b)$ with initial vorticity $\omega^0$ and initial circulation $\gamma \in \R$. This solution enjoys a Hodge decomposition, its circulation is conserved in time and $ \|\omega \|_{L^{\infty}(\R_{+}\times \Omega)}\leq \|\omega^0\|_{L^{\infty}(\Omega)}$.
\end{theorem}
The solutions in Theorem~\ref{thm:existencenonsmooth} are constructed by compactness when we approximate lakes as in Definition~\ref{def-generallake} by a sequence of smooth lakes. For such weak solutions, we need to define a generalized notion of tangency and circulation, see \eqref{imperm} and \eqref{eq:gencirc} later.

Uniqueness of these solutions, called ``weak interior solutions'', is not known to hold. If in addition the lake is assumed to be smooth ($b$ and the boundary are $C^3$), it is proved in Theorem 2.3 in \cite{BM} (see Proposition 2.12 in \cite{LNP} for not simply connected domains) that a Calder\'on-Zygmund-type inequality holds for the solution of the div-curl elliptic problem \eqref{eq:elliptic} below. Therefore, in this case, $v$ is continuous up to the boundary and we can define the tangency condition $v\cdot n=0$ and the circulation $\int_{\partial \cal I}v\cdot\tau$ in the classical sense. Moreover, these additional regularity properties allow one to adapt the proof of Yudovich to obtain the uniqueness of global weak solution in the class $(v,\omega)\in L^\infty_{\loc}(\R_{+};L^2(\Omega))\times L^\infty_{\loc}(\R_{+} \times \Omega)$, see \cite{BM,LNP}. For smooth lakes, the $L^\infty$-norm of the vorticity and the circulation (Kelvin's theorem) are conserved.

\subsubsection*{The evanescent island} 
 Our first main result characterizes the limit as $\eps\rightarrow 0$ for a vanishing island (see Assumption~\ref{ass:evanescent}). 
As the domains depend on $\varepsilon$, all functions are intended to be defined on $\Omega$ by extension to $0$, in particular on the island $\Ieps$. Our first main result reads as follows:
\begin{theorem}\label{thm:maingeneral}
Let $(\Weps,\beps)$ be a sequence of lakes with one non-degenerated island (as in Definition~\ref{def-generallake}), which converges (in the sense of Assumption~\ref{ass:evanescent}) to a punctured lake $(\Omega,b)$ (as in Definition~\ref{ass:b}). Given $\gamma\in \R$ and $\omega^0\in L^{\infty}(\Omega)$, let $(\veps,\weps)$ be a global weak solution given in Theorem~\ref{thm:existencenonsmooth} with initial vorticity $\omega^0$ and circulation $\gamma \in \R$. As $\eps\rightarrow 0$, there exists a subsequence, still denoted by $(\veps,\weps)$, such that
 \begin{equation*}
\begin{aligned}
 \sqrt{\beps}\veps&\rightarrow \sqrt{b}v \quad \text{strongly in} \quad L^{2}_{\loc}(\R_{+};L^2(\Omega)),\\
 \weps&\rightharpoonup^{\ast} \omega \quad \text{weakly-$\ast$ in} \quad L^{\infty}(\R_{+}\times \Omega),
\end{aligned}
\end{equation*}
where $(v,\omega)$ is a solution of the vorticity formulation in $\R_{+}\times \Omega$, in the sense of Definition~\ref{defi:vorticity}, except that
\[
\curl v=b\omega+\gamma\delta_{0} \quad \text{in} \quad \mathcal{D}'([0,\infty)\times \Omega).
\]
In addition, $v$ satisfies the Hodge-decomposition \eqref{eq:Hodgelimit} and $v\in L^{\infty}(\R_{+};L_{\loc}^p(\Omega))$ with $\frac{1}{p}=\frac{1}{2}+\frac{1}{q}$, where $q\in[2,\frac{4}{a_1})$. If further $a_0\in[0,1)$, then $v \in L^{\infty}(\R_{+};L^p(\Omega))$.
 \end{theorem}
Several remarks are in order. 
\begin{enumerate}
\item The initial velocity field $\veps^0$ is uniquely determined by $\omega^0\Big|_{\Weps}$ and the circulation $\gamma$ around $\cal I_{\varepsilon}$ by virtue of Proposition~\ref{prop:BS}. Similarly, the initial data $v^0$ for the limiting lake is uniquely determined by $\omega^0$ and $\gamma$, see Corollary~\ref{coro:BSlimit}.
\item As in \cite{LNP}, when we consider a sequence of lakes where $\Omega_{\varepsilon}\subset\Omega$, we can state that the weak formulation of the vorticity equation \eqref{eq:vorticity} holds true for any $\Phi\in C_c^{\infty}(\R_+\times\overline{\Omega})$. A byproduct of this theorem is then a global existence result for the limit system with such test functions.
\item For the solutions under consideration in the present paper, additional difficulties arise close to the origin where $b$ vanishes, hence providing a suitable weak formulation of the velocity equation turns out to be a difficult task. 
Actually, the issues encountered are reminiscent to the ones occurring for compressible fluids in the presence of vacuum regions \cite{BJ}. 
In Section~\ref{sec:velocity}, we obtain the asymptotic velocity equation for \eqref{eq:lake-vel} posed on the lake $(\Oo,b)$ and discuss possibly strategies for the velocity formulation in $(\Omega,b)$.
\item Condition {\em (3)} in Definition~\ref{ass:b} arises naturally to ensure that $\Cap_{b^{-1}}(\{0\})>0$, see Lemma~\ref{lem:capacity}, where the weighted capacity is defined as
\begin{equation*}
 \Cap_{b^{-1}}(\{0\})=\inf_{\varphi\in\mathcal{B}(\{0\})}\int_{\Omega}\frac{1}{b}|\nabla\varphi|^2\dd x
\end{equation*}
with $\mathcal{B}(\{0\})=\{\varphi\in C_c^{\infty}(\Omega) : \varphi=1 \, \text{in a small neighborhood of }\, 0 \}$.

We recall that the unweighted $H^1$-capacity of a single point vanishes. To require that an obstacle has positive $H^1$-capacity is similar to assuming that the connected compact set is constituted by at least two points. The notion of positive capacity was exploited for the 2D-Euler equations \cite{GV-Lacave1} on singular domains and the lake equations \cite{LNP} to obtain non-erasable obstacles. Working with the weighted capacity constitutes a refinement of this analysis that allows us to consider islands collapsing to a single point for \eqref{eq:lake-vel}. 

Actually, the positive weighted capacity of $\{0\}$ suggests that it could be sufficient to consider the target system posed on the punctured lake $(\Omega\setminus\{0\},b)$, i.e. where test functions are supported in $\Omega\setminus\{0\}$, see Proposition~\ref{prop:puncturedlake} and Theorem~\ref{thm:velocity}, where it suffices to assume $a_{1}> 0$. We refer to Remark~\ref{rem:genass} for a deeper discussion on the generalized assumptions that include more singular geometries.

We mention that the weight $b^{-1}$ is a Muckenhoupt weight \cite{M} under Definition~\ref{ass:b} if in addition $a_0\in[0,1)$. There is an extensive literature on weighted Sobolev capacities, we refer the reader e.g. to the monograph \cite[Section 2]{HKM} and references therein. In \cite{Munteanu}, Munteanu considered degenerated lakes for possibly vanishing topographies assuming that the depth function $b$ is of class $C^2$ up to the boundary and a Muckenhoupt weight. These assumptions rule out power law depth functions $b=|x|^{a_1}$ and allow only for logarithmic depth functions such as $\Omega=B(0,5/2)\setminus\overline{B(0,1)}$ and $b=\log|x|\log(3-|x|)$.

\end{enumerate}

\begin{example}
To illustrate the effect of the degenerating topography, we consider the radial case with $\Omega=B(0,1)$ and $b$ radial such as $b(x)=|x|^{a_1}$. Then, we have the same harmonic function $H$ as for the Euler equations, namely the unique function satisfying
\begin{equation*}
\diver(b H)=0 \text{ in } \Omega, \quad b H\cdot \mb n=0 \text{ on } \partial\Omega, \quad \curl H = \delta_{0}
\end{equation*}
is $H= \frac{1}{2\pi}\frac{x^\perp}{|x|^2}$ and $\sqrt{b}H\in L^2(\Omega)$. Let $\beps=|x|^{a_1}-\eps$, $\Weps=B(0,1)\cap\{\beps>0\}$, hence the sequence of islands is given by $\mathcal{I}_\eps=\overline{B}(0,\eps^{1/a_{1}})$ and shrinks homothetically to the origin. Then, $H_{\varepsilon}=H\vert_{\Omega_{\varepsilon}}$ is the unique solution to
\begin{equation}\label{eq:Heps}
\diver(b_{\varepsilon} H_{\varepsilon})=\curl H_{\varepsilon}= 0 \text{ in } \Omega_{\varepsilon}, \ b_{\varepsilon} H_{\varepsilon}\cdot \mb n=0 \text{ on } \partial\Omega_{\varepsilon}, \ \oint_{\partial \cal I_{\varepsilon}} H_{\varepsilon}\cdot \mb \tau \dd s = 1.
\end{equation}
One then verifies that $\sqrt{\beps}H_\eps\in L^2(\Weps)$ uniformly bounded. We remark that the respective stream function, namely the function $\psi_\eps\in H^1_0(\Omega_\eps)$ such that $\nabla^{\perp}\psi_\eps=b_\eps H_\eps$, is uniformly bounded in $H^1(\Weps)$, whereas for the Euler equations we have $\psi_\eps=\frac1{2\pi}\ln |x|\vert_{\Omega_\eps}$. 
\end{example}

This suggests that the degeneracy of $b$ has a desingularizing effect on the respective stream function. We are thus led to introduce uniform estimates for $\sqrt{b_\eps}v_\eps=\sqrt{b_\eps}^{-1}\nabla\psi_\eps^{\perp}$, see Section~\ref{sec:BS}. This corresponds to studying the respective stream functions in weighted Sobolev spaces. This is quite natural, since the energy associated to \eqref{eq:lake-vel} yields a $L^2$-bound for $\sqrt{b_\eps}v_\eps$. The vanishing of $b_\eps$ close to the origin is pivotal for our method as it turns out to desingularize the respective stream functions. While the small obstacle problem is $L^2$-critical for the 2D Euler equations (see \cite{GV-Lacave1,ILL1}), we show that it is subcritical for the lake equations under suitable assumptions on the degeneracy of $b$.

When the bottom is flat and the domain $\Omega_\eps$ takes the form $\R^2\setminus \mathcal{I}_\eps$, the authors in \cite{ILL1} managed to solve the elliptic problem \eqref{eq:Heps} for general geometries via conformal mapping. In general, it is not clear how to adapt such an approach for the lake equations \eqref{eq:lake-vel}, but we refer to \cite{DekeyserVanS} for an interesting integral representation of the Green kernel for \eqref{eq:Heps} provided that $b$ does not vanish. The paper \cite{DekeyserVanS} also studies the presence of a point vortex in the lake equations and proves that the center of a Dirac mass moves in the direction of $\nabla^\perp b$, but without diffuse vorticity ($\omega\equiv 0$ therein). A byproduct of Theorem~\ref{thm:maingeneral} is the existence of a point vortex with a diffuse part, where the point vortex is stuck to the origin, where $\nabla b=0$ provided that $b$ is continuously differentiable. Here, $b(0)=0$ and the stability of the point vortex under a uniformly increasing water level yielding $b(0)>0$ is an interesting problem. 

Finally, we mention that uniqueness of weak solutions is a delicate open question for the lake equations posed on non-smooth lakes. Indeed, if the topography vanishes on the boundary, uniqueness is only known for smooth lakes \cite{BM, LNP}.
 It is not clear whether the Calder\'on-Zygmund type inequalities introduced in \cite{BM} can be recovered when $b(0)=0$. Therefore, the uniqueness of weak solutions to our limit system for both velocity and vorticity formulation seems to be a difficult open question, even when $\gamma=0$. When $\gamma\neq0$, the weak solution considered in Theorem~\ref{thm:maingeneral} contains a Dirac mass in the vorticity. Even for the Euler equations, uniqueness in the presence of a Dirac mass in the vorticity is only established when $\omega^{0}$ is constant in a neighborhood of point vortex (see \cite{lacave-miot} for more details). Concerning the linear case, we mention \cite{BJ}, where a strategy for uniqueness of weak solutions to linear advection equations of type \eqref{eq:transport} is pointed out.

\subsubsection*{The emergent island}

\begin{theorem}\label{thm:main2general}
Let $(\Weps,\beps)$ be a sequence of lakes without island (as in Definition~\ref{def-generallake}), which converges (in the sense of Assumption~\ref{ass:emergent}) to a punctured lake $(\Omega,b)$ (as in Definition~\ref{ass:b}). Given $\omega^0\in L^{\infty}(\Omega)$, let $(\veps,\weps)$ be a global weak solution provided by Theorem~\ref{thm:existencenonsmooth} with initial vorticity $\omega^0$. As $\eps\rightarrow 0$, there exists a subsequence, still denoted by $(\veps,\weps)$, such that
\begin{equation*}
\begin{aligned}
 \sqrt{\beps}\veps&\rightarrow \sqrt{b}v \quad \text{strongly in} \quad L^{2}_{\loc}(\R_{+};L^2(\Omega)),\\
 \weps&\rightharpoonup^{\ast} \omega \quad \text{weakly-$\ast$ in} \quad L^{\infty}(\R_{+}\times \Omega),
\end{aligned}
\end{equation*}
where $(v,\omega)$ is a solution of the vorticity formulation in $\R_{+}\times \Omega$, in the sense of Definition~\ref{defi:vorticity}. In addition, $v\in L^{\infty}(\R_{+};L_{\loc}^p(\Omega))$ with $\frac{1}{p}=\frac{1}{2}+\frac{1}{q}$, where $q\in[2,\frac{4}{a_1})$. If further $a_0\in[0,1)$, then $v\in L^{\infty}(\R_{+};L^p(\Omega))$.
\end{theorem}
We remark that $\curl v=b\omega\in L^{\infty}$ does not include a Dirac mass in contrast to Theorem~\ref{thm:maingeneral}. 
As $(\Weps,\beps)$ is simply connected for $\eps>0$, the velocity field is uniquely determined by $\weps$ and the proof of Theorem~\ref{thm:main2general} simplifies substantially compared to the one of Theorem~\ref{thm:maingeneral}.

\medskip

Combining our two main theorems, the question of continuity when $b_{\varepsilon}=b\pm \varepsilon$ and $\eps\to 0$ is then totally solved when the circulation $\gamma$ is equal to zero. 

A related interesting question for $\gamma \neq 0$ would be to study the occurence of an island as in Theorem~\ref{thm:main2general}, starting from a sort of vortex-wave system \cite{Marchioro-Pulvirenti} for the lake equations. 

Putting together our analysis with the stability result in \cite{LNP}, a possible extension could cover the case of multiple islands with one or several shrinking to a point or appearing (because this island is the lowest), as it was considered in \cite{L07} for the 2D Euler equations.

\medskip

The remainder of this article is divided in four sections. Section~\ref{sec:main} introduces the Hodge type decomposition and provides uniform estimates for the stream functions. Section~\ref{sec:evanescent} is dedicated to the proof of Theorem~\ref{thm:maingeneral}, while Section~\ref{sec:emergent} provides the proof of Theorem~\ref{thm:main2general}. Section~\ref{sec:velocity} addresses the asymptotic velocity formulation. Finally, Appendix~\ref{app:limit} introduces the density of smooth test functions in the weighted space.

\bigskip
\noindent
{\bf Acknowledgements.} This work is supported by the French National Research Agency in the framework of the ``Investissements d'avenir'' program (ANR-15-IDEX-02) and of the project ``SINGFLOWS'' (ANR-18-CE40-0027-01). E.M. acknowledges the project ``INFAMIE'' (ANR-15-CE40-01).

\section{Setup of the mathematical framework and uniform bounds}\label{sec:main}

\subsection{Lakes without degenerated island and elliptic problems}\label{sec:recall}

In this section, we set up the mathematical framework for the study the lake equations, we refer to \cite{LNP} for full details and complete proofs. In the following, we consider a lake $(\Omega_{\varepsilon},b_{\varepsilon})$ as in Definition~\ref{def-generallake} with at most one (non-degenerated) island and we recall that we denote by $\Weps=\Tweps\setminus \mathcal{I}_{\eps}$ in the case of one island.

As for the 2D-Euler equations, the main ingredient is to use conserved quantities for the vorticity. To that end, one requires the reconstruction of the velocity field $\veps$ in terms of the vorticity $\weps$, through the following div-curl problem
\begin{equation}\label{eq:divcurl}
\diver(\beps\veps)=0 \text{ in } \quad \Weps, \quad \curl\veps=\beps\weps \text{ in } \Weps, \quad \beps\veps\cdot \mb n=0 \text{ on } \partial\Weps .
\end{equation}
The second equation reads in $\mathcal{D}'(\Omega_{\varepsilon})$, whereas the first and third ones have to be understood in the following weak sense:
\begin{equation}\label{imperm}
 \int_{\Omega_{\varepsilon}} \beps(x) \veps(x) \cdot h(x)\dd x = 0, 
\end{equation}
for any test function $h$ in the function space $G(\Omega_{\varepsilon})$ defined by
\begin{equation*}
G (\Omega_{\varepsilon}):=\Big \{w\in L^2(\Omega_{\varepsilon}) \ : \ w=\nabla p, \ \text{ for some } p\in H^1_{\loc}(\Omega_{\varepsilon}) \Big \}.
\end{equation*}
For $\beps \veps \in L^2(\Omega_{\varepsilon})$, such a condition in \eqref{imperm} is equivalent to
\begin{equation*}
\beps \veps \in \mathcal{H}(\Omega_{\varepsilon}) ,
\end{equation*}
where 
\begin{equation*}
\mathcal{H}(\Omega_{\varepsilon})= \text{ the closure in $L^2$ of } \{ \varphi \in C_c^\infty(\Omega_{\varepsilon}) \ | \ \diver \varphi =0 \}. 
\end{equation*}
This equivalence can be found, for instance, in \cite[Lemma III.2.1]{Galdi} where it was pointed out that if $\Omega_{\varepsilon}$ is a regular bounded domain and if $\beps \veps$ is a sufficiently smooth function, then $\beps \veps$ verifies \eqref{imperm} if and only if $\diver(\beps\veps) = 0$ and $\beps \veps \cdot \mb n\vert_{\partial \Omega_{\varepsilon}} = 0$.

Similarly to \eqref{imperm}, the weak form of the divergence free and tangency conditions on $\beps \veps\in L^{2}(\R_{+}\times \Omega_{\varepsilon})$ also reads: 
\begin{equation} \label{imperm2}
\forall h \in C_c^\infty\left([0,+\infty); G(\Omega_{\varepsilon})\right), \quad \int_{\R_+} \int_{\Omega_{\varepsilon}} \beps(x) \veps(t,x) \cdot h(t,x) \, \dd x\dd t = 0.
\end{equation}

 The first and third conditions in \eqref{eq:divcurl} imply that the div-curl problem is related to the following elliptic problem on the stream function $\psi_{\varepsilon}\in H^1_{0}(\widetilde{\Omega}_{\varepsilon})$ 
 \begin{equation}\label{eq:elliptic}
 \diver\Big(\frac{1}{\beps}\nabla\psi_{\eps} \Big)=f_{\eps} \text{ in } \Weps, \quad \partial_{\tau}\psi_{\eps} =0 \text{ on } \partial\cal I_{\varepsilon},
\end{equation}
with the relation $\beps\veps = \nabla^\perp \psi_{\varepsilon}$ and where $f_{\varepsilon}$ will be $b_{\varepsilon}\omega_{\varepsilon}$. When the depth function is allowed to vanish at the beaches, \eqref{eq:elliptic} is not uniformly elliptic. Here and below, we refer to an elliptic problem as degenerated whenever it lacks uniform ellipticity.

To solve this problem, we first introduce the similar elliptic problem with Dirichlet boundary condition on the island $\mathcal{I}_\eps$ and the outer boundary
\begin{equation}\label{eq:elliptic0}
 \diver\Big(\frac{1}{\beps}\nabla\psi_{\eps}^0 \Big)=f_{\eps}\quad \text{in} \quad \Weps, \qquad \psi_{\varepsilon}\in H^1_{0}({\Omega}_{\varepsilon}),
\end{equation}
for which existence and uniqueness is studied in \cite{LNP} in the following space
\begin{equation}\label{eq:Xeps}
X_{\beps}(\Weps)=\left\{\psi \in H_0^1(\Weps):\ \quad \frac{1}{\sqrt{\beps}}\nabla \psi\in L^2(\Weps)\right\}.
\end{equation}
For shortness, we will often write $X_\eps$ instead to $X_{\beps}(\Weps)$. The function space $X_b(\Omega)$ for the limiting topography is defined analogously. We recall that $X_{\beps}(\Weps)$ is a Hilbert space with inner product
\begin{equation*}
 \left\langle f,g\right\rangle_{X_{\eps}}=\int_{\Weps}\frac{1}{\beps}\nabla f \cdot \nabla g\,\dd x.
\end{equation*}
One of the most important properties showed in \cite{LNP} is the density of $C_c^{\infty}(\Weps)$ in $X_{\eps}$. Such a result was stated there for smooth lakes without degenerated island. We improve the proof in Appendix~\ref{app:limit} (in particular Lemma~\ref{HardyLem}) to remark that the density property also holds for lakes considered in Definitions~\ref{def-generallake} and \ref{ass:b}. The main consequence of the density result is the uniqueness of solutions for \eqref{eq:elliptic0} and for the following Hodge decomposition. The density property in the punctured lake will be also used for the compactness argument in Section~\ref{sec:emergent}. 

Hence, for a simply connected lake, there exists a unique solution in $X_{\eps}$ of \eqref{eq:elliptic}. However, when the domain $\Omega_{\varepsilon}$ is not simply connected (i.e. with islands), the vorticity is not sufficient to uniquely determine the velocity field, namely to obtain uniqueness in \eqref{eq:elliptic}. We have to analyze functions of harmonic type.

\begin{definition}\label{defi:bepsharmonic}
A function $\Phi_{\eps}$ is called $\beps$-harmonic if
\begin{equation*}
 \Phi_{\eps}\in H_0^1(\widetilde{\Omega}_{\varepsilon}), \qquad \frac{1}{\sqrt{\beps}}\nabla\Phi_{\eps}\in L^2(\Weps),
\end{equation*}
is solution to
\begin{equation*}
 \diver\Big(\frac{1}{\beps}\nabla\Phi_\eps\Big)=0 \quad \text{in} \quad \Weps, \qquad \partial_\tau\Phi_\eps=0 
 \quad \text{on} \quad \partial \Omega_{\varepsilon}.
\end{equation*}
The space of $\beps$-harmonic functions in $\Omega_{\varepsilon}$ is denoted $\mathcal{H}_{\beps}(\Weps)$ (or $\mathcal{H}_{\eps}$). 
\end{definition}

The condition $\partial_{\tau} \Phi =0$ should be understood as the existence of a constant $C$ such that $\Phi-C\chi_{\delta}\in H^1_{0}(\Omega_{\varepsilon})$ where $\chi_{\delta}$ is a smooth cutoff function as in \eqref{eq:chidelta} below.

It was proved in \cite[Prop. 2.5]{LNP} that $\mathcal{H}_{\beps}(\Weps)$ is a vector space of dimension equal to the number of islands, which is one when we consider the case of the evanescent island, see Theorem~\ref{thm:maingeneral}. A basis is formed by the function 
\begin{equation}\label{def-phi1}
 \phi^1_{\varepsilon} \in \mathcal{H}_{\beps}(\Weps)\quad \text{such that} \quad \phi^1_{\varepsilon}\vert_{\partial\cal I_{\varepsilon}} \equiv 1.
\end{equation}
Another basis was also constructed in terms of the circulation. For that, we introduce the notion of generalized circulation. Let $\chi_{\delta}$ be a smooth cut-off function such that
\begin{equation}\label{eq:chidelta}
 \chi_{\delta}\in C_c^{\infty}(\widetilde{\Omega}_\eps), \quad \chid=1 \text{ in } B(0,\delta), \quad \chi_\delta=0 \text{ away from } B(0,2\delta).
\end{equation}
In all the sequel, $\delta>0$ is fixed such that for $\varepsilon$ small enough $\cal I_{\varepsilon}\subset B(0,\delta)\subset B(0,2\delta)\subset \Omega$. For a lake $(\Weps,\beps)$ as in Definition~\ref{def-generallake} with one island $\mathcal{I}_\eps$ and $\veps$ solution to \eqref{eq:divcurl} on $(\Weps,\beps)$ with $\sqrt{\beps}\veps\in L^2(\Weps)$ and $\curl \veps\in L^1(\Weps)$, the generalized circulation around the island $\cal I_{\varepsilon}$ is defined as 
\begin{equation}\label{eq:gencirc}
 \gamma(\veps):=-\int_{\Weps}\chid \curl \veps\dd x-\int_{\Weps}\nabla^{\perp}\chi_{\delta}\cdot\veps\dd x.
\end{equation}
The circulation $\gamma$ is well-defined as $\veps\in L^2(\supp(\nabla\chid))$ and independent from the choice of $\chi_{\delta}$, see \cite[Appendix A]{LNP}. When $\veps$ is sufficiently regular, this definition reads as the standard circulation $ \gamma(\veps)=-\int_{\Weps}\curl(\chi_{\delta}\veps)\dd x =\int_{\Weps}\diver(\chi_{\delta}\veps^{\perp})\dd x=\oint_{\partial \cal I_{\varepsilon}} \veps \cdot \tau$ where the last integral is considered in the counterclockwise direction, i.e. $\tau=-\mb n^\perp$. Then we introduce $\psi^1_{\varepsilon}$ such that
\begin{equation}\label{def-psi1}
 \psi^1_{\varepsilon} \in \mathcal{H}_{\beps}(\Weps)\quad \text{such that} \quad \gamma\left(\frac{1}{\beps}\nabla^{\perp}\pun\right)=1.
\end{equation}

We are now in position to give the decomposition of the velocity field $\veps$ by means of stream functions in the spirit of a Hodge decomposition, see Proposition 2.10 in \cite{LNP}.
\begin{proposition}\label{prop:BS}
Let $(\Weps,\beps)$ be a lake with one island $\mathcal{I}_\eps$ as in Definition~\ref{def-generallake}, $\gamma\in \R$ and $f_{\varepsilon}\in L^{2}(\Weps)$. Then there exists a unique vector field $\veps$ such that $\sqrt{\beps}\veps\in L^2(\Weps)$ with
\begin{equation*}
 \diver(\beps\veps)=0 \text{ in } \Weps, \quad \beps\veps\cdot \mb n=0 \text{ on } \partial\Weps,\quad \text{(in the sense of \eqref{imperm})}
\end{equation*}
and 
\begin{equation*}
 \curl\veps=f_{\varepsilon} \quad \text{in} \quad \mathcal{D}'(\Weps), \quad \gamma(\veps)=\gamma,
\end{equation*}
where $\gamma(\veps)$ denotes the generalized circulation of $\veps$ introduced in \eqref{eq:gencirc}. Moreover, one has a Hodge type decomposition
\begin{equation}\label{eq:BS}
 \veps=\frac{1}{\beps}\nabla^{\perp}\psi_{\eps}^0+\alphae\frac{1}{\beps}\nabla^{\perp}\pun,
\end{equation}
where $\psi_{\eps}^0$ is the unique solution in $X_{\eps}$ to \eqref{eq:elliptic0}, $\pun$ is the unique function verifying \eqref{def-psi1} and
\begin{equation*}
 \alphae=\gamma(\veps)+\int_{\Weps}f_{\varepsilon}\phie\dd x,
\end{equation*}
where $\phie$ is uniquely defined by \eqref{def-phi1}.
\end{proposition}
The main consequence is that $\veps^0$ is uniquely defined by prescribing the potential vorticity $\omega^0$ and the circulation $\gamma$.
Such a result is also valid for lakes $(\Weps,\beps)$ as in Definition~\ref{def-generallake} without island, for which $\veps=\frac{1}{\beps}\nabla^{\perp}\psi_{\eps}^0$ and no notion of circulation is needed to uniquely determine the velocity field. 
 
\begin{remark}\label{rem:genass2}
The lakes considered in Definition~\ref{def-generallake} are slightly smoother than in \cite{LNP} because we have assumed (\emph{3})-(\emph{4}) in order to ensure the density of $C_c^{\infty}(\Weps)$ in $X_{\eps}$. Nevertheless, by the compactness argument in \cite{LNP}, we already know the existence part of the Hodge decomposition, whereas the uniqueness part is not necessary for the sequel.

In the same way, Definition~\ref{ass:b} allows us to state that $v^{0}$ is uniquely determined from $\omega^{0}$ and $\gamma$ (see Corollary~\ref{coro:BSlimit}). If we are not interested by the uniqueness of the elliptic problems, we can notice that our compactness analysis will give the existence of the Hodge decomposition without the need of (\emph{4}) in Definition~\ref{ass:b}. For a formulation in $\Oo$ we could even expect to generalize (\emph{3}) in Definition~\ref{ass:b}.
\end{remark}

\subsection{Uniform bounds}\label{sec:BS}

In what follows, all functions defined on $\Weps$ are intended to be extended by zero to the domain $\Omega$.

By the property of the solutions constructed in Theorem~\ref{thm:existencenonsmooth}, we directly deduce that the vorticity is uniformly bounded, independently from $\varepsilon$:
\begin{equation}\label{bd:vorticity}
 \|\weps \|_{L^{\infty}(\R_{+}\times \Weps)}\leq \|\omega^0\|_{L^{\infty}(\Weps)} \leq \|\omega^0\|_{L^{\infty}(\Omega)}.
\end{equation}
Next, we provide uniform $X_{\eps}$-bounds for the stream functions appearing in \eqref{eq:BS}, where the space $X_{\eps}$ is defined in \eqref{eq:Xeps}. For that purpose, we exploit several times that $C_c^{\infty}(\Weps)$ is dense in $X_{\eps}$ w.r.t the norm $\|\cdot\|_{X_{\eps}}$. For a smooth lake, this density property is proven in \cite[Lem. 2.1]{LNP}. In Appendix~\ref{app:limit}, we show that the respective density still holds on lakes characterized by Definition~\ref{def-generallake}, see Proposition~\ref{prop density}. The density result includes lakes with non-degenerated islands as in Definition~\ref{def-generallake} and punctured lakes, as well as vanishing or non-vanishing topographies at the outer boundary.

In both problems (evanescent and emergent islands), we prove a uniform bound of the solution $\psi^0_{\varepsilon}$ to \eqref{eq:elliptic0}. 
\begin{lemma}\label{lem:elliptic}
Let $(\Weps,\beps)$ be a sequence of lakes with at most one non-degenerated island (as in Definition~\ref{def-generallake}), which converges (in the sense of Assumption~\ref{ass:evanescent} or \ref{ass:emergent} respectively) to a punctured lake $(\Omega,b)$ (as in Definition~\ref{ass:b}). There exists $C=C(\Omega,b )>0$ such that for all $\eps>0$, the unique solution $\psi_{\eps}^0\in X_{\eps}$ to \eqref{eq:elliptic0} with $f_{\varepsilon}=\beps\omega_{\varepsilon}$ satisfies 
\[
\|\psi_{\eps}^0\|_{L^\infty(\R_{+}; H^1_{0}(\Omega_{\varepsilon}))}+ \Big\|\frac{1}{\sqrt{\beps}}\nabla\psi_{\eps}^0\Big\|_{L^\infty(\R_{+}; L^2(\Weps))}\leq C \|\omega^{0}\|_{L^\infty(\Omega)}.
\]
\end{lemma}

\begin{proof} With $f_{\varepsilon}=\beps \omega_{\varepsilon}$, \eqref{eq:elliptic0} is satisfied in the sense of distributions for any test function compactly supported in $\Omega_{\varepsilon}$. For a lake in the sense of Definition~\ref{def-generallake}, $C^\infty_{c}(\Omega_{\varepsilon})$ is dense in $X_{\eps}$, see Proposition~\ref{prop density}, so we can choose $\psi_{\varepsilon}^0(t,\cdot)$ as a test function in the weak formulation. This gives for a.e. $t\in \R_+$
\begin{equation*}\begin{split}
 \int_{\Weps}\left|\frac{1}{\sqrt{\beps}}\nabla\psi_{\eps}^0(t,x)\right|^2\dd x&=-\int_{\Weps} \beps \omega_{\varepsilon}(t,x)\psi_{\eps}^0(t,x)\dd x\\
&\leq \|\beps \omega_{\varepsilon}(t,\cdot)\|_{L^2(\Weps)}\|\psi_{\eps}^0(t,\cdot)\|_{L^2(\Weps)}.\end{split}
\end{equation*}
Choosing $\Omega^0=B(0,R_{0})$ large enough such that $\Omega_{\varepsilon}\cup\Omega \subset \Omega^0$ for all $\varepsilon$, the Poincar{\'e} inequality applied on the domain $\Omega^0$ yields that there exists $C_{0}=C(\Omega^0)>0$ such that
\begin{equation*}\begin{split}
 \|\psi_{\eps}^0(t,\cdot)\|_{L^2(\Weps)}= \|\psi_{\eps}^0(t,\cdot)\|_{L^2(\Omega^0)}&\leq C_{0} \|\nabla \psi_{\eps}^0(t,\cdot)\|_{L^2(\Omega^0)}= C_{0}\|\nabla\psi_{\eps}^0(t,\cdot)\|_{L^2(\Weps)} \\
 &\leq C_{0} \| b_{\varepsilon}\|_{L^\infty}^{1/2} \left\|\frac{1}{\sqrt{\beps}}\nabla\psi_{\eps}^0(t,\cdot)\right\|_{L^2(\Weps)}.
\end{split}
\end{equation*}
It follows from \eqref{bd:vorticity}, that there exists $C_{1}=C(C_0,|\Omega^0|)>0$ such that
\begin{equation*}
\left\|\frac{1}{\sqrt{\beps}}\nabla\psi_{\eps}^0(t,\cdot)\right\|_{L^2(\Weps)}\leq C_{1} \| b_{\varepsilon}\|_{L^\infty(\Weps)}^{3/2}\|\omega^{0}\|_{L^\infty(\Weps)}.
\end{equation*}
The Poincar\'e inequality just above then implies
\begin{equation*}
\|\psi_{\eps}^0(t,\cdot) \|_{H^1_{0}(\Omega_{\varepsilon})} \leq C_{1} \| b_{\varepsilon}\|_{L^\infty(\Weps)}^{2}\|\omega^{0}\|_{L^\infty(\Weps)}.
\end{equation*}
The uniform bound of $\beps$ (see {\em (2)} in Assumptions~\ref{ass:evanescent} or \ref{ass:emergent} respectively) ends the proof of this lemma.
\end{proof}

 For the emergent island, this lemma is enough to derive the uniform estimate for $\sqrt\beps\veps=\frac{1}{\sqrt{\beps}}\nabla^\perp \psi_{\eps}^0$. 
\begin{corollary}\label{lem:veps1}
Let $(\Weps,\beps)$ be a sequence of lakes without island (as in Definition~\ref{def-generallake}), which converges (in the sense of Assumption~\ref{ass:emergent}) to a punctured lake $(\Omega,b)$ (as in Definition~\ref{ass:b}). Given $\omega^0\in L^{\infty}(\Omega)$, let $(\veps,\weps)$ be a global weak solution given by Theorem~\ref{thm:existencenonsmooth} with initial vorticity $\omega^0$. There exists $C=C(\Omega,b )>0$ such that for all $\eps>0$
\begin{equation*}
 \left\|\sqrt{\beps}\veps\right\|_{L^{\infty}(\R_{+};L^2(\Weps))}\leq C\|\omega^0\|_{L^{\infty}(\R_+\times \Weps)}.
\end{equation*}
\end{corollary}

For the evanescent island, we have to analyze $\beps$-harmonic functions.
The main novelty of this section are uniform estimates on $\phi^1_{\varepsilon}$ and $\psi^1_{\varepsilon}$ in $X_{\eps}$ independent of the size of the island. We mention such estimates are not available for a flat topography $b=const$, namely for the 2D-Euler equations \cite{ILL1}.
\begin{example}\label{ex:1}
Let $\Omega=B(0,1)$, $r_{\eps}>0$ such that $r_{\eps}\rightarrow 0$ as $\eps\rightarrow 0$, and $\Ieps=B(0,r_\eps)$, i.e. $\Weps=B(0,1)\backslash B(0,r_\eps)$. First, consider $\beps$ independent from $x$ such that $\beps\rightarrow b$ with $b>0$. Then,
\[
\phi^1_{\varepsilon}= \frac{\ln |x|}{\ln r_{\varepsilon}} \qquad \text{and}\qquad \psi^1_{\varepsilon} =\frac{\beps\ln |x|}{2\pi},
\]
where $\phie$ and $\pun$ solve \eqref{def-phi1} and \eqref{def-psi1} respectively. We notice that
\begin{equation*}
 \|\nabla\phie\|_{L^2(\Weps)}\rightarrow 0, \qquad \|\nabla\psi_\eps^1\|_{L^2(\Weps)}\rightarrow \infty,
\end{equation*}
as $\eps\rightarrow 0$. The small obstacle problem for the $2D$-Euler equations is $L^2$-critical in terms of velocity and $H^1$-critical in terms of stream functions \cite{ILL1}. Second, we consider $b(x)=|x|^{\alpha}$ for $\alpha\in(0,2)$. A straightforward computation, see e.g. \cite[Section 2.23]{HKM}, yields the explicit representations of $b$-harmonic functions on $\Weps$:
\begin{equation*}
 \phie=\frac{|x|^{\alpha}-1}{r_{\eps}^{\alpha}-1} \qquad \text{and}\qquad \psi_\eps^1=\frac{|x|^{\alpha}-1}{2\pi\alpha}.
\end{equation*}
One has that 
\begin{equation*}
 \left\|\frac{1}{\sqrt{\beps}}\nabla\phie\right\|_{L^2(\Weps)}=\sqrt{\frac{2\pi\alpha}{1-r_\eps^{\alpha}}} \qquad \text{and}\qquad \left\|\frac{1}{\sqrt{\beps}}\nabla\pun\right\|_{L^2(\Weps)}=\sqrt{\frac{(1-r_{\eps}^{\alpha})}{2\pi\alpha}}.
\end{equation*}
\end{example}
The example showcases the effect of the degeneracy: for radially symmetric and vanishing topography $b$, the $b$-harmonic stream functions are uniformly bounded in weighted spaces. 

We begin by providing a uniform estimate for the sequence of $\beps$-harmonic functions $\phi^1_{\varepsilon}$ defined by \eqref{def-phi1} for lakes satisfying Assumption~\ref{ass:evanescent}.

\begin{lemma}\label{lem:uniformphieps}
Let $(\Weps,\beps)$ be a sequence of lakes with one non-degenerated island (as in Definition~\ref{def-generallake}), which converges (in the sense of Assumption~\ref{ass:evanescent}) to a punctured lake $(\Omega,b)$ (as in Definition~\ref{ass:b}). There exists $C=C(\Omega,b )>0$ such that for all $\eps>0$ sufficiently small, the unique solution $\phi_{\eps}^1$ to \eqref{def-phi1} satisfies 
\[
\|\phi_{\eps}^1\|_{H^1(\Omega_{\varepsilon})}+ \Big\|\frac{1}{\sqrt{\beps}}\nabla\phi_{\eps}^1\Big\|_{L^2(\Weps)}\leq C .
\]
\end{lemma}

The proof of the uniform bound is inspired by Lemma 3.2 in \cite{LNP} where the existence of $\phi^1_{\varepsilon}$ was proved. We reproduce the main steps to show the $\eps$-independence of the bounds albeit the degeneracy in $0$.

\begin{proof}
Fix $\delta\in(0,1)$ and let $\chid\in C_c^{\infty}(\Omega)$ be as in \eqref{eq:chidelta}. There exists $\eps_0>0$ such that for $\eps\in(0,\eps_0)$, one has that $\chid\equiv 1$ in a neighbourhood of $\Ieps$.
Define
\begin{equation}\label{eq:decompositionphi}
 \phie=\tilde\phie+\chid,
\end{equation}
where $\tilde\phie\in X_{\eps}$ is the solution to
\begin{equation}\label{eq:tildephie}
 \diver\left(\frac{1}{\beps}\nabla\tilde{\phie}\right)=- \diver\left(\frac{1}{\beps}\nabla\chid\right) \text{ in } \mathcal{D}'(\Weps).
\end{equation}
Existence and uniqueness of $\tilde{\phie}$ follows from \cite[Proposition 2.1]{LNP}. Indeed, for regular $\beps\in H_{\loc}^1(\Weps)$, we notice that (\emph{2}) of Definition~\ref{def-generallake} ensures that $\beps$ is bounded from below on $\supp(\nabla\chid)$ as $\supp(\nabla\chid)\Subset \Weps$ for $\eps>0$ sufficiently small, hence $\diver(\beps^{-1}\nabla\chid)\in L^2(\Weps)$ which allows to solve directly \eqref{eq:tildephie}. For less regular $\beps\in L^\infty(\Weps)$, existence was obtained by compactness when we approximate lakes as in Definition~\ref{def-generallake} by a sequence of smooth lakes. Uniqueness is clear as long as $C_c^{\infty}(\Weps)$ is dense in $X_{\eps}$, see Proposition~\ref{prop density}. 

Exploiting again this density of $C_c^{\infty}(\Weps)$ in $X_{\eps}$, we may use $\tilde{\phie}$ as test-function in \eqref{eq:tildephie}. Hence,
\begin{equation}\label{eq:Abel-1}
\left\|\frac{1}{\sqrt{\beps}}\nabla\tilde{\phie}\right\|_{L^2(\Weps)}^2
=-\int_{\Weps}\frac{1}{\sqrt{\beps}}\nabla\tilde{\phie}\cdot\frac{1}{\sqrt{\beps}}\nabla\chid\dd x.
\end{equation}
The Cauchy-Schwarz inequality then yields
\begin{equation*}
\left\|\frac{1}{\sqrt{\beps}}\nabla\tilde{\phie}\right\|_{L^2(\Weps)}\leq \left\|\frac{1}{\sqrt{\beps}}\nabla\chid\right\|_{L^2(\Weps)}.
\end{equation*}
The bound follows from the uniform lower bound for $\beps$ in (\emph{3}) Assumption~\ref{ass:evanescent} and $\chid\in C_c^\infty(\Omega)$ with $\supp(\nabla\chid)\subset\Oo$.
\end{proof}

Next, we provide a uniform estimate for the $\beps$-harmonic functions $\psi^1_{\varepsilon}$ with constant circulation. The uniform bound will follow from the observation that the $b^{-1}$-capacity of the limiting island $\{0\}$ is positive. For a compact set $E\subset \Omega$, the weighted capacity is defined as
\begin{equation*}
 \Cap_{b^{-1}}(E,\Omega)=\inf_{\varphi\in \mathcal{B}(E)}
 \int_{\Omega}\frac{1}{b}|\nabla\varphi|^2\dd x,
\end{equation*}
where $\mathcal{B}(E)=\{\varphi\in C_c^{\infty}(\Omega) : \varphi= 1 \, \text{in a neighborhood of} \, E \}$.

\begin{lemma}\label{lem:capacity}
Let $(\Omega,b)$ a punctured lake as in Definition~\ref{ass:b}. Then, 
\begin{equation*}
 \Cap_{b^{-1}}(\{0\},\Omega)>0.
\end{equation*}
Let $(\Omega_{\varepsilon},\beps)$ a lake with one non-degenerated island as in Definition~\ref{def-generallake}. For any $\delta>0$, there exists $C_{\delta}>0$ such that the following holds: if $(\Omega_{\varepsilon},\beps)$ verifies
\[
\Omega_{\varepsilon}\subset B(0,\delta^{-1}), \ \| \beps\|_{L^\infty(\Omega_{\varepsilon})}\leq \delta^{-1}, \ \int_{\Omega_{\varepsilon}\cap B(0,\delta)} \frac{\beps(y)}{|y|^2}\dd y\leq \delta^{-1}
\]
then 
\begin{equation*}
 \Cap_{\beps^{-1}}(\mathcal{I}_{\varepsilon},\widetilde{\Omega}_{\eps})\geq C_{\delta}.
\end{equation*}
\end{lemma}

\begin{remark}\label{rem:capacity}
We notice that any sequence of lakes $(\Weps,\beps)$ with one non-degenerated island (as in Definition~\ref{def-generallake}), which converges in the sense of Assumption~\ref{ass:evanescent} to a punctured lake $(\Omega,b)$ (as in Definition~\ref{ass:b}), satisfies the assumptions of this lemma. Indeed, for $\delta$ small enough such that $\Omega\Subset B(0,\delta^{-1})$, (\emph{1}) of Assumption~\ref{ass:evanescent} implies that the inclusions hold true for $\varepsilon$ small enough. The uniform estimate of $\beps$ comes directly from (\emph{2}). The last property can be also derived from (\emph{2}) and the fact that $a_{1}>0$, see Definition~\ref{ass:b}.
\end{remark}

\begin{proof}[Proof of Lemma~\ref{lem:capacity}]
Let $\varphi\in \mathcal{B}(\{0\})$. Upon extending $\varphi$ by $0$ on $\R^2\backslash\Omega$, one has the formula
\[
 \varphi(x)=\frac{1}{2\pi}\int_{\R^2}\frac{x-y}{|x-y|^2}\cdot\nabla\varphi(y)\dd y,
\]
for all $x\in \R^2$, see \cite[Lemma 7.14]{GT}. Therefore,
\begin{align*}
4\pi^2&=4\pi^2\varphi(0)^2=\left(\int_{\text{supp}(\nabla\varphi)}\frac{y}{|y|^2}\cdot\nabla\varphi(y)\dd y\right)^2\\
&\leq \left(\int_{\Omega}\frac{1}{{b(y)}}\left|\nabla\varphi(y)\right|^2\dd y\right)\left(\int_{\Omega}\frac{b(y)}{|y|^2}\dd y\right).
\end{align*}
In view of Definition~\ref{ass:b} there exist $a_1\in(0,2)$ and $\delta>0$ such that $b(x)\leq C|x|^{a_1}$ for all $x\in B(0,\delta)$ and
\begin{equation*}
 \int_{\Omega}\frac{b(x)}{|x|^2}\dd x\leq C\int_{B(0,\delta)}|x|^{a_1-2}\dd x+\int_{\Omega\backslash B(0,\delta)}\frac{b(x)}{|x|^{2}}\dd x\leq \frac{C}{a_1} \delta^{a_1}+\frac{C(\|b\|_{L^{\infty}},|\Omega|)}{\delta^2}.
\end{equation*}
Finally,
\begin{equation*}
 \frac{4\pi^2}{C\left(\frac{\delta^{a_1}}{a_1}+\frac{1}{\delta^2}\right)}\leq \int_{\Omega}\frac{1}{b}\left|\nabla\varphi\right|^2\dd x.
\end{equation*}
Taking the infimum over all $\varphi\in \mathcal{B}(\{0\})$ completes the proof of the first statement.

Reproducing this proof with $\beps$ and $\Omega_{\varepsilon}$ is straightforward, because it is clear that the important assumption is that $\int_{\Omega_{\varepsilon}\cap B(0,\delta)} \frac{\beps(y)}{|y|^2}\dd y$ is bounded independently from $\varepsilon$.
\end{proof}

We are now in position to prove the uniform estimates for the $\beps$-harmonic functions defined in \eqref{def-psi1}.
\begin{lemma}\label{lem:bepsharmonic}
Let $(\Weps,\beps)$ be a sequence of lakes with one non-degenerated island (as in Definition~\ref{def-generallake}), which converges (in the sense of Assumption~\ref{ass:evanescent}) to a punctured lake $(\Omega,b)$ (as in Definition~\ref{ass:b}). There exists $C=C(\Omega,b )>0$ such that for all $\eps>0$, the unique solution $\pun$ to \eqref{def-psi1} satisfies 
\[
\|\pun\|_{H^1(\Omega_{\varepsilon})}+ \Big\|\frac{1}{\sqrt{\beps}}\nabla\pun\Big\|_{L^2(\Weps)}\leq C .
\]
\end{lemma}

\begin{proof}
We recall that due to the vector space structure of $\mathcal{H}_{\beps}(\Weps)$, there exists $\{a_{\eps}\}_{\eps>0}\subset \R$ such that 
\begin{equation*}
 \pun=a_{\eps}\phie.
\end{equation*}
In view of Lemma~\ref{lem:uniformphieps} it suffices to show that $a_{\eps}$ is bounded. To that end, we notice that the generalized circulation $\gamma$, defined in \eqref{eq:gencirc} satisfies
\begin{equation}\label{eq:circulation-aeps}
\begin{aligned}
 1=\gamma\left(\frac{1}{\beps}\nabla^{\perp}\pun\right)
 &=-a_\eps\int_{\Weps}\frac{1}{\beps}\nabla\chid\cdot\nabla\phie\dd x\\
 &=-a_\eps\int_{\Weps}\frac{1}{\beps}\left|\nabla\phie\right|^2\dd x,
 \end{aligned}
\end{equation}
where we have used that $\chid =\phie-\Tilde{\phie}$, see \eqref{eq:decompositionphi}, and the fact that $C_c^{\infty}(\Weps)$ is dense in $X_{\eps}$ allows us to consider the equation $\diver(\beps^{-1}\nabla\phie)=0$ tested against $\tphie\in X_{\eps}$.
In virtue of \eqref{eq:circulation-aeps}, Lemma~\ref{lem:uniformphieps} can be interpreted as a lower bound of $-a_\eps$. Next, we seek an upper bound. For that purpose it suffices to show that there exists $c>0$ such that 
\begin{equation}\label{eq:lower bound}
 \inf_{\eps>0}\int_{\Weps}\frac{1}{\beps}\left|\nabla\phie\right|^2\dd x\geq c>0,
\end{equation}
and the bound follows from \eqref{eq:circulation-aeps}. As $\tphie\in X_{\eps}$ there exists a sequence $\varphi_\eps^n\in C_c^{\infty}(\Weps)$ such that $\varphi_\eps^n$ converges strongly to $\tphie$ in $X_{\eps}$. It follows that $\phi_{\eps}^{1,n}:=\varphi_\eps^n+\chid$ converges strongly to $\phie$ in $X_\varepsilon$ and that 
\[
\int_{\Weps}\frac{1}{\beps}\left|\nabla\phie\right|^2\dd x = \lim_{n\rightarrow\infty}\int_{\Weps}\frac{1}{\beps}\left|\nabla\phi_{\eps}^{1,n}\right|^2\dd x.
\]
Moreover, $\phi_{\eps}^{1,n}=1$ in a neighborhood of $\mathcal{I}_{\varepsilon}$ because $\chid=1$ on $B(0,\delta/2)$ and $\varphi_\eps^n=0$. The bound then follows directly from Lemma~\ref{lem:capacity} and Remark~\ref{rem:capacity}:
\begin{equation*}
\int_{\Weps}\frac{1}{\beps}\left|\nabla\phie\right|^2\dd x\geq \Cap_{\beps^{-1}}(\mathcal{I}_{\varepsilon},\widetilde{\Omega}_{\eps}) \geq c>0.
\end{equation*}
Therefore, we obtain from \eqref{eq:lower bound} and \eqref{eq:circulation-aeps} that $-a_{\eps}$ is uniformly bounded from above, as desired. 
\end{proof}

We are now ready to combine the previous lemmas with the conservation of the circulation and estimates on the vorticity \eqref{bd:vorticity} in the decomposition \eqref{eq:BS} in order to infer uniform bounds on $\sqrt{\beps}\veps$ for emergent island.

\begin{lemma}\label{lem:veps2}
Let $(\Weps,\beps)$ be a sequence of lakes with one non-degenerated island (as in Definition~\ref{def-generallake}), which converges (in the sense of Assumption~\ref{ass:evanescent}) to a punctured lake $(\Omega,b)$ (as in Definition~\ref{ass:b}). Given $\gamma\in \R$ and $\omega^0\in L^{\infty}(\Omega)$, let $(\veps,\weps)$ be a global weak solution given by Theorem~\ref{thm:existencenonsmooth} with initial vorticity $\omega^0$ and circulation $\gamma \in \R$. There exists $C=C(\Omega,b )>0$ such that for all $\eps>0$
\begin{equation*}
 \left\|\sqrt{\beps}\veps\right\|_{L^{\infty}(\R_{+};L^2(\Weps))}\leq C\left(|\gamma| +\|\omega^0\|_{L^{\infty}(\R_+\times \Weps)}\right).
\end{equation*}
\end{lemma}

Here, we notice that this estimate is not sufficient to infer uniform $L^p$ estimates on the velocity field $\veps$ (for $p>1$) as this would require uniform estimates on $\sqrt{\beps}^{-1}$ in $L^q(\Weps)$ with $q>2$ that we lack unless $a_0\in[0,1)$, see Definition~\ref{ass:b}.

\begin{proof}
The Hodge decomposition \eqref{eq:BS} allows one to write
\begin{equation*}
\sqrt{\beps} \veps(t,\cdot)=\frac{1}{\sqrt{\beps} }\nabla^{\perp}\psi_{\eps}^0(t,\cdot)+\alphae(t)\frac{1}{\sqrt{\beps} }\nabla^{\perp}\pun,
\end{equation*}
where the first right hand side term is estimated in Lemma~\ref{lem:elliptic}. In view of Lemma~\ref{lem:bepsharmonic}, it is suffices to prove that $\alphae$ is uniformly bounded. We observe that
\begin{equation*}
 \|\alphae(t)\|_{L^{\infty}(\R_+)}\leq |\gamma|+ C_R\|\beps\|_{L^{\infty}(\Weps)}\|\weps\|_{L^{\infty}(\R_+\times\Weps)}\|\phi_{\eps}^1\|_{L^2(\Weps)},
\end{equation*}
being bounded in virtue of \eqref{bd:vorticity} and Lemma~\ref{lem:uniformphieps}. This completes the proof of Lemma~\ref{lem:veps2}.
\end{proof}

\begin{remark}\label{rem:genass}
The degeneracy of $b$ in $0$ is characterised by $\Cap_{b^{-1}}(\{0\})>0$. If $\Cap_{b^{-1}}(\{0\})=0$, the analysis is expected to be similar to the small obstacle problem for flat topographies \cite{ILL1}. On the other hand, for simplicity and because it is enough for the application that we have in mind, we assume $\beps\leq Cb$ and $\Omega_{\varepsilon}\subset \Omega$ in Assumption~\ref{ass:evanescent}. One may want to generalize the assumptions on the shrinking process.

We note that all the arguments in this section hold true if we replace (\emph{1})-(\emph{2}) of Assumption~\ref{ass:evanescent} by 
\begin{enumerate}
\item[\emph{(1)}] ${\Omega}_{\varepsilon}\to \Omega\backslash\{0\}$ in the Hausdorff sense;
 \item[\emph{(2i)}] $\beps\in L^\infty(\Omega_{\varepsilon})$ uniformly bounded;
 \item[\emph{(2ii)}] there exists $\delta >0$ such that $\displaystyle \int_{\Omega_{\varepsilon}\cap B(0,\delta)} \frac{\beps(y)}{|y|^2}\dd y\leq \delta^{-1}$.
\end{enumerate}
In particular, Lemma~\ref{lem:capacity} gives a uniform estimate of the weighted capacity of the island, which is the main ingredient of our analysis.

In the following section, we will use $\beps\leq Cb$ and $\Omega_{\varepsilon}\subset \Omega$ to compare the $X_{\varepsilon}$ and $X$ norms for all $\varphi$ supported in $\Omega_{\varepsilon}$:
\[
\| \varphi \|_{X} = \int_{\Omega} \frac{|\nabla \varphi|^2}b = \int_{\Omega_{\varepsilon}} \frac{|\nabla \varphi|^2}b \leq C \| \varphi \|_{X_{\varepsilon}} .
\]
Such an inequality will give elegant proofs for the compactness. Nevertheless we think that the following analysis could be adapted for the weaker assumptions listed in this remark. For instance, we do not use such an inequality in Section~\ref{sec:emergent}.

We mention several possible generalizations of our theorems. First, it was noted in Remark~\ref{rem:genass2} that the density of $C_c^{\infty}$ in $X_\eps$ is not mandatory to state the existence of a Hodge decomposition. This density was also used in this section to infer that some equations in the sense of distributions are also true tested by some particular functions belonging to $X_{\varepsilon}$, such as $\tilde\phi^1_{\varepsilon}$. However, such a property could be also obtained by compactness in the construction made in \cite{LNP}. Hence, one may remove (\emph{3})-(\emph{4}) in Definition~\ref{def-generallake}. Second, it will be noticed later that we could allow $a_{1}\geq 2$ in Definition~\ref{ass:b} if we are only interested by the limit formulation in $\Oo$. In particular, we have not used $a_{1}<2$ in Section~\ref{sec:main}. At the opposite, assuming that (\emph{2ii}) above is not verified entails the loss of the uniform lower bound $\Cap_{\beps^{-1}}(\Ieps,\Weps)\geq c>0$. By consequence, $a_{\varepsilon}\to -\infty$ in the proof of Lemma~\ref{lem:bepsharmonic}. To prove a local estimate of $v_{\varepsilon}$ in $\Oo$ would require to prove that $\phi^1_{\varepsilon}$ converges locally to zero. At least in the radial case, see Example~\ref{ex:1}, such a convergence may not be expected. An example of lakes satisfying the relaxed assumptions but not (2ii) is given by
\begin{equation*}
 \Weps=B(0,1)\setminus B(0,\eps), \quad \beps:=\begin{cases}
 0 \quad x\in B(0,\eps),\\ 1 \quad x\in B(0,\sqrt{\eps})\setminus\overline{B(0,\eps)}, \\ |x|^{a_1} \quad x\in \Omega\setminus\overline{B(0,\sqrt{\eps})}.
 \end{cases}
\end{equation*}
\end{remark}

\section{Compactness for the evanescent island}\label{sec:evanescent}
The aim of this section is to prove Theorem~\ref{thm:maingeneral}. To that end, we first develop an $L^2$-based stability theory for the sequence $\sqrt{\beps}\veps$ by exploiting the Hodge decomposition \eqref{eq:BS}.

\subsection{Compactness for stream functions}
We prove stability for the degenerate elliptic equation \eqref{eq:elliptic}, where we need to deal simultaneously with the singularity of the geometry and the topography. 

We recall that $\Weps\subset \Omega$ for all $\eps>0$, see Assumption~\ref{ass:evanescent} and notice that 
\begin{equation}\label{eq:beps}
 \beps\rightarrow b \quad \text{in} \quad L_{\loc}^p(\Omega),
\end{equation}
for all $p\in[1,\infty)$ from Assumption~\ref{ass:evanescent} upon extending $\beps$ by zero to $\Omega$ and interpolating with the uniform $L^{\infty}$-bounds for $\beps$ and $b$. In addition, let $K\subset \Oo$. As $\Weps$ converges to $\Oo$ in Hausdorff sense, there exists $\eps_0>0$ such that for all $\eps\in(0,\eps_0]$ one has $K\subset \Weps$. Therefore,
\begin{equation}\label{eq:sqrtbeps}
 \left\|\frac{1}{\sqrt{\beps}}-\frac{1}{\sqrt{b}}\right\|_{L_{\loc}^p(\Oo)}= \left\|\frac{\beps-b}{\sqrt{\beps}\sqrt{b}(\sqrt{\beps}+\sqrt{b})}\right\|_{L_{\loc}^p(\Oo)}\rightarrow 0,
\end{equation}
for all $p\in[1,\infty)$.

In the following, $\frac{1}{\sqrt{\beps}}\nabla\tilde{\phie}$ and $\tilde{\phie}$ are extended by zero to $\Omega$. 

\begin{lemma}\label{lem:convergence-phie}
Let $(\Weps,\beps)$ be a sequence of lakes with one non-degenerated island (as in Definition~\ref{def-generallake}), which converges (in the sense of Assumption~\ref{ass:evanescent}) to a punctured lake $(\Omega,b)$ (as in Definition~\ref{ass:b}). Let $\phie$ be the unique $\beps$-harmonic function solution to \eqref{def-phi1}. Then, up to a subsequence
\[
 \frac{1}{\sqrt{\beps}}\nabla\phie\rightarrow \frac{1}{\sqrt{b}}\nabla\phi^1 \qquad \text{strongly in } \quad L^2(\Omega),
\]
where $\phi^1\in X_b(\Omega)$ is such that $\diver(b^{-1}\nabla\phi^1)=0$ in $\mathcal{D}'(\Oo)$.
\end{lemma}

\begin{proof}
We recall from Lemma~\ref{lem:uniformphieps} that $\phie$ is the unique $\beps$-harmonic function on $\Weps$ such that $\phie=1$ on $\partial \mathcal{I}_{\eps}$, in the sense that $\phie$ satisfies the decomposition
\begin{equation*}
 \phie=\Tilde{\phie}+\chi_{\delta},
\end{equation*}
with $\chid$ as in \eqref{eq:chidelta} and $\Tilde{\phie}\in X_{\eps}$ solution to \eqref{eq:tildephie}. Upon extending $\tilde{\phie}$ by zero to $\Omega$, we conclude from Lemma~\ref{lem:uniformphieps} that $\tilde{\phie}\in H^1_{0}(\Omega)$ uniformly bounded. Therefore, there exists $\tilde{\phi^1}\in H^1_0(\Omega)$ such that $\tilde{\phie}\rightharpoonup \tilde{\phi^1}$ in $H^1_0(\Omega)$ up to passing to subsequences. Here, we have used in a crucial way that $\Omega_{\varepsilon}\subset \Omega$ to infer the Dirichlet boundary condition for $\tilde{\phi^1}$ on $\partial\Omega$. In the next section, we will need the notion of $\gamma$ and Mosco's convergences.

Let $f_{\eps}:={\sqrt{\beps}^{-1}}\nabla\tilde{\phie}$ and extend $f_\eps$ by zero to $\Omega$. Since $f_\eps\in L^2(\Omega)$ uniformly bounded from Lemma~\ref{lem:uniformphieps}, there exists $f\in L^2(\Omega)$ such that $f_{\eps}\rightharpoonup f$ in $L^2(\Omega)$ up to subsequences. We wish to identify the weak limit as $f={\sqrt{b}^{-1}}\nabla\tilde{\phi}^1$ a.e. on $\Omega$. To that end, let $\varphi\in C_c^{\infty}(\Oo)$. The Hausdorff-convergence of $\Weps$ to $\Oo$ in Hausdorff sense yields that $\supp(\varphi)\subset\Weps$ for $\eps$ sufficiently small. One has
\begin{equation*}
 \left\langle \frac{1}{\sqrt{\beps}}\nabla\tilde\phie, \varphi\right\rangle\rightarrow \left\langle\frac{1}{\sqrt{b}}\nabla\tilde{\phi}^1,\varphi \right\rangle,
\end{equation*}
since $\sqrt{\beps}^{-1}$ converges strongly to $\sqrt{b}^{-1}$ in $L_{\loc}^q(\Oo)$ for some $q\geq 2$, see \eqref{eq:sqrtbeps} and $\nabla\Tilde{\phie}$ converges weakly in $L^2(\Omega)$. Since $f$ and $\sqrt{b}^{-1}\nabla\tilde{\phi}^1$ are functions belonging to $L^2(\Omega)$, we conclude that they coincide a.e. on $\Omega$. By uniqueness of this limit, we do not need to extract a subsequence in the weak limit $f_{\eps}\rightharpoonup f$. Such an identification procedure will be used several times in the sequel.

Finally, $\tilde{\phi}^1\in X_b(\Omega)$. In addition, we conclude that
\begin{equation*}
\frac{1}{\sqrt{\beps}}\nabla\phie\rightharpoonup \frac{1}{\sqrt{b}}\nabla\phi^1:=\frac{1}{\sqrt{b}}\nabla(\tilde{\phi}^1+\chid) \quad \text{in} \quad L^2(\Omega).
\end{equation*}
Next, we pass to the limit in the equation verified by $\phie$ (see Definition~\ref{defi:bepsharmonic}). As $\phie$ is $\beps$-harmonic in $\Weps$ for $\eps>0$, and $\Weps$ converges to $\Oo$ in Hausdorff sense, one has
\begin{equation*}
 0=\left\langle\diver\left(\frac{1}{\beps}\nabla\phie\right),\varphi\right\rangle,
\end{equation*}
for all $\varphi\in C_c^{\infty}(\Oo)$ and $\eps$ sufficiently small.
We conclude that 
\begin{multline*}
 0=\left\langle\diver\left(\frac{1}{\beps}\nabla\phie\right),\varphi\right\rangle=-\int_{\Omega}\frac{1}{\sqrt{\beps}}\nabla\phie\cdot\frac{1}{\sqrt{\beps}}\nabla\varphi \dd x\\
 \rightarrow -\int_{\Omega}\frac{1}{\sqrt{b}}\nabla\phi^1\cdot\frac{1}{\sqrt{b}}\nabla\varphi \dd x
 =\left\langle\diver\left(\frac{1}{b}\nabla\phi^1\right),\varphi\right\rangle,\notag
\end{multline*}
for any $\varphi\in C_c^{\infty}(\Oo)$ as ${\sqrt{\beps}^{-1}}{\sqrt{\beps}^{-1}}\nabla{\phie}$ forms a weak-strong pair converging weakly in $L_{\loc}^1(\Oo)$ in view of \eqref{eq:sqrtbeps}. It follows that
\begin{equation*}
\int_{\Omega}\frac{1}{b}\nabla\tilde{\phi}^1\cdot\nabla\varphi\dd x=-\int_{\Omega}\frac{1}{b}\nabla\chid\cdot\nabla\varphi \dd x
\end{equation*}
for all $\varphi\in C_c^{\infty}(\Oo)$. In particular,
\begin{equation}\label{eq:tildephi1}
\int_{\Omega}\frac{1}{b}\nabla\tilde{\phi}^1\cdot\nabla\tphie\dd x=-\int_{\Omega}\frac{1}{b}\nabla\chid\cdot\nabla\tphie \dd x.
\end{equation}
Indeed, it follows from Proposition~\ref{prop density} that there exists $\varphi_\eps^n\in C_c^{\infty}(\Weps)$ such that $\varphi_\eps^n$ converges to $\tphie$ in $X_{\eps}$ as $n\rightarrow \infty$. Since $\Weps\subset \Omega$ for all $\eps>0$ from (\emph{1}) Assumption~\ref{ass:evanescent}, one has $\varphi_\eps^n\in C_c^{\infty}(\Omega\setminus \{0\})$ and
\begin{equation*}
\limsup_{n\rightarrow \infty}\left\|\frac{1}{\sqrt{b}}(\nabla\varphi_{\eps}^n-\nabla\tphie)\right\|_{L^2(\Omega)}\leq \limsup_{n\rightarrow \infty}C \left\|\frac{1}{\sqrt{\beps}}(\nabla\varphi_{\eps}^n-\nabla\tphie)\right\|_{L^2(\Weps)}=0,
\end{equation*}
where we used (\emph{2}) Assumption~\ref{ass:evanescent}. 

As $\|\sqrt b^{-1}\nabla\tphie \|_{L^2(\Omega)}\leq C\|\sqrt \beps^{-1}\nabla\tphie \|_{L^2(\Omega_{\varepsilon})} \leq C$, we have that $\sqrt b^{-1}\nabla\tphie$ converges weakly to $\sqrt b^{-1}\nabla \tilde{\phi}^1$ in $L^2(\Omega)$ and we may now pass to the limit in \eqref{eq:tildephi1} yielding that
\begin{equation}\label{eq:bharmonic}
 \int_{\Omega}\frac{1}{b}\left|\nabla\Tilde{\phi}^1\right|^2\dd x=-\int_{\Omega}\frac{1}{b}\nabla\chid\cdot\nabla\Tilde{\phi}^1\dd x,
\end{equation}
which is equivalent to \eqref{eq:Abel-1} when the island is non-degenerated.
Since
\begin{equation*}\begin{split}
\int_{\Weps}\frac{1}{\beps}\nabla\chid\cdot\nabla\Tilde{\phie}\dd x=
 \int_{\Omega}\frac{1}{\beps}\nabla\chid\cdot\nabla\Tilde{\phie}\dd x
 \rightarrow \int_{\Omega}\frac{1}{b}\nabla\chid\cdot\nabla\Tilde{\phi}^1\dd x,\end{split}
\end{equation*} 
identities \eqref{eq:Abel-1} and \eqref{eq:bharmonic} allow us to conclude
\begin{equation*}
 \int_{\Weps}\frac{1}{\beps}\left|\nabla\tilde{\phie}\right|^2\dd x \to \int_{\Omega}\frac{1}{b}\left|\nabla\Tilde{\phi}^1\right|^2\dd x.
\end{equation*}
Finally, we have shown that 
\begin{equation*}
 \frac{1}{\sqrt{\beps}}\nabla\Tilde{\phie}\rightarrow \frac{1}{\sqrt{b}}\nabla\tilde\phi^1 \qquad \text{strongly in} \quad L^2(\Omega).
\end{equation*}
We obtain
\begin{equation*}
 \frac{1}{\sqrt{\beps}}\nabla{\phie}\rightarrow \frac{1}{\sqrt{b}}\nabla\phi^1 \qquad \text{strongly in} \quad L^2(\Omega),
\end{equation*}
with $\phi^1\in X_b(\Omega)$ and $\diver(b^{-1}\nabla\phi^1)=0$ in $\mathcal{D}'(\Oo)$.
\end{proof}

Next, we show strong compactness for $\beps$-harmonic functions with constant circulation, namely $\pun$ solution to \eqref{def-psi1}.
\begin{lemma}\label{lem:bharmonic}
Let $(\Weps,\beps)$ be a sequence of lakes with one non-degenerated island (as in Definition~\ref{def-generallake}), which converges (in the sense of Assumption~\ref{ass:evanescent}) to a punctured lake $(\Omega,b)$ (as in Definition~\ref{ass:b}). Let $\pun$ be the unique $\beps$-harmonic function such that $\gamma(\frac{1}{\beps}\nabla^\perp\psi_\eps^1)=1$. Then, up to a subsequence
\begin{equation*}
 \frac{1}{\sqrt{\beps}}\nabla\pun\rightarrow \frac{1}{\sqrt{b}}\nabla \psi^1, \qquad \text{strongly in} \quad L^2(\Omega),
\end{equation*}
where $\psi^1\in X_b(\Omega)$ with $\diver(b^{-1}\nabla\psi^1)=0$ in $\mathcal{D}'(\Oo)$ and $\gamma(\frac{1}{b}\nabla^\perp\psi^1)=1$.
\end{lemma}

\begin{proof}
For $\eps>0$, the vector space structure of $\mathcal{H}_{\beps}(\Weps)$ implies that there exists $a_{\eps}\in \R$ such that $\psi_\eps^1=a_\eps\phie$ with $\phie\in \mathcal{H}_{\beps}(\Weps)$ unique solution to \eqref{def-phi1}. In the proof of Lemma~\ref{lem:bepsharmonic}, we have proved that there exists $C>0$ such that $C^{-1}\leq -a_\eps \leq C$ for every $\varepsilon>0$. Hence, in virtue of the Bolzano-Weierstrass Theorem, there exists a convergent subsequence, still denoted by $a_\eps$, converging to $a\in (-\infty,0)$. 
From Lemma~\ref{lem:convergence-phie}, we conclude that
\begin{equation*}
\frac{1}{\sqrt{\beps}}\nabla\psi_\eps^1=\frac{a_\eps}{\sqrt{\beps}}\nabla\phie\rightarrow\frac{a}{\sqrt{b}}\nabla\phi^1,
\end{equation*}
strongly in $L^2(\Omega)$. Upon defining $\psi^1=a\phi^1$, we obtain that $\sqrt{\beps}^{-1}\nabla\psi_\eps^1$ converges strongly to $\sqrt{b}^{-1}\nabla\psi^1$ in $L^2(\Omega)$. 

Passing to the limit in the definition of the generalized circulation \eqref{eq:gencirc}, we easily conclude that $\gamma(\frac{1}{b}\nabla^\perp\psi^1)=1$.
\end{proof}
It remains to show compactness for the sequence $\psi_\eps^0$ solution to \eqref{eq:elliptic0} with Dirichlet conditions.

By \eqref{bd:vorticity}, the Banach-Alaoglu Theorem states that there exists $\omega\in L^{\infty}(\R_+\times \Omega)$ such that $\weps\rightharpoonup^{\ast}\omega$ in $L^{\infty}(\R_+\times \Omega)$ upon extending $\weps$ by zero on $\Omega\backslash\Weps$ and up to passing to a subsequence. In particular, it follows from \eqref{eq:beps} that
\begin{equation}\label{eq:convergence bepsweps}
 \sqrt{\beps}\weps\rightharpoonup^{\ast} \sqrt{b}\omega, \qquad {\beps}\weps\rightharpoonup^{\ast} b\omega \qquad \text{in} \quad L^{\infty}(\R_+;L^p(\Omega))
\end{equation}
for all $p\in(1,\infty)$ and up to extracting subsequences.

\begin{lemma}\label{lem:Dirichlet}
Let $(\Weps,\beps)$ be a sequence of lakes with one non-degenerated island (as in Definition~\ref{def-generallake}), which converges (in the sense of Assumption~\ref{ass:evanescent}) to a punctured lake $(\Omega,b)$ (as in Definition~\ref{ass:b}). Let $\psi_{\eps}^0\in X_{\eps}$ be such that
\begin{equation}\label{eq:psi0eps}
 \diver(\frac{1}{\beps}\nabla\psi_\eps^0)=\beps\weps, \qquad \text{in} \quad \mathcal{D}'(\R_{+}\times \Weps),
\end{equation}
where $\beps\weps$ satisfies \eqref{eq:convergence bepsweps}.
Then, there is a subsequence $\varepsilon\to0$ such that
\begin{equation*}
 \frac{1}{\sqrt{\beps}}\nabla\psi_{\eps}^0\rightarrow \frac{1}{\sqrt{b}}\nabla \psi^0 \qquad \text{in} \quad L_{\loc}^2(\R_+;L^2(\Omega)),
\end{equation*}
where $\psi^0\in L^{\infty}(\R_+;X_b(\Omega))$ such that for a.e. $t\in \R_+$
\begin{equation*}
 \diver\left(\frac{1}{b}\nabla\psi^0(t,\cdot)\right)=b\omega(t,\cdot) \qquad \text{in} \quad \mathcal{D}'(\Oo).
\end{equation*}
\end{lemma}

\begin{proof}
We recall that Lemma~\ref{lem:elliptic} yields that $\psi_{\eps}^0\in L^{\infty}(\R_+;H_0^1(\Weps))$ uniformly bounded and $\psi_\eps^0\in L^{\infty}(\R_+;X_{\eps})$ uniformly bounded. Further, it follows from \eqref{eq:vorticity} and \eqref{eq:psi0eps} that 
\begin{equation*}
 \diver\left(\frac{1}{\beps}\nabla\partial_t\psi_{\eps}^0\right)=\partial_t(\beps\weps)=-\diver(\beps\veps\weps)
\end{equation*}
in the sense of distributions. Now, we claim that $\partial_t\psi_\eps^0\in L^{\infty}(\R_+;X_{\eps})$. Indeed, the solutions constructed in \cite{BM,LNP} were obtained by compactness where at the limit $\psi_{\eps}^0\in W^{1,\infty}(\R_+;X_{\eps})$. Thus, exploiting that $C_c^{\infty}(\Weps)$ is dense in $X_{\eps}$, we can consider $\partial_t\psi_\eps^0(t,\cdot)$ as a test function in the previous equation to state that for a.e. $t\in\R_{+}$
\[
 \left\|\frac{1}{\sqrt{\beps}}\nabla\partial_t\psi_\eps^0(t,\cdot)\right\|_{L^2(\Weps)}^2 =-\int_{\Omega_{\varepsilon}} \beps\veps(t,x)\weps \cdot \nabla\partial_t\psi_\eps^0(t,x)\dd x
\]
hence
\begin{equation*}
\left\|\frac{1}{\sqrt{\beps}}\nabla\partial_t\psi_\eps^0\right\|_{L^{\infty}(\R_+;L^2(\Weps))}\leq \|\sqrt{\beps}\veps\|_{L^{\infty}(\R_{+};L^2(\Weps))}\|\beps\weps\|_{L^{\infty}(\R_+;L^{\infty}(\Weps))}
\end{equation*}
yielding $\partial_t\psi_\eps^0\in L^{\infty}(\R_+;X_{\eps})$ uniformly bounded. In particular, by applying the Poincar\'e inequality in $\Omega$ we have $\psi_{\eps}^0\in W^{1,\infty}(\R_+;H_0^1(\Weps))$ uniformly bounded. Upon extending $\psi_{\eps}^0$ and $\frac{1}{\sqrt{\beps}}\nabla\psi_{\eps}^0$ by zero on $\Omega$ and recalling that the sequence $\Weps\subset \Omega$ we recover that there exists $\psi^0\in W^{1,\infty}(\R_+;H_0^1(\Omega))$ and a subsequence still denoted by $\varepsilon\to 0$ such that 
\begin{equation*}
 \psi_{\eps}^0\rightharpoonup^{\ast} \psi^0 \quad \text{in} \quad W^{1,\infty}(\R_+;H_0^1(\Omega)), \qquad \psi_{\eps}^0\rightarrow \psi^0 \quad \text{in} \quad C(([0,T];L^2(\Omega)),
\end{equation*}
for any $T>0$. Moreover, there exists $f\in W^{1,\infty}(\R_+;L^2(\Omega))$ such that 
\begin{equation*}
 \frac{1}{\sqrt{\beps}}\nabla\psi_\eps^0\rightharpoonup^{\ast} f \quad \text{in} \quad W^{1,\infty}(\R_+;L^2(\Omega)),
\end{equation*}
where we can show that $f=\frac{1}{\sqrt{b}}\nabla\psi^0$ a.e. in $\R_+\times\Omega$ proceeding as in the proof of Lemma~\ref{lem:convergence-phie}.
Therefore, $\psi^0\in W^{1,\infty}(\R_+;X_b(\Omega))$. Actually, the previous proof is the same as in \cite{LNP} to show that $\psi^0_{\varepsilon}\in W^{1,\infty}(\R_+;X_\varepsilon)$.

Next, we wish to pass to the limit in \eqref{eq:psi0eps}. Let $\varphi\in C_c^{\infty}(\R_+\times\Oo)$, then there exists $T>0$ such that $\supp(\varphi)\subset [0,T)\times\Weps$ for all $\eps$ sufficiently small as $\Weps\rightarrow \Oo$ in Hausdorff sense. Since $\beps\weps$ converges weakly-$\ast$ to $b\omega$ in $L^{\infty}(\R_+\times \Omega)$ from \eqref{eq:convergence bepsweps}, we obtain that
\begin{equation*}
 \diver\left(\frac{1}{b}\nabla\psi^0\right)=b\omega \qquad \text{in} \quad \mathcal{D}'(\R_{+}\times\Oo).
\end{equation*}
hence, for a.e. $t\in \R_+$
\begin{equation*}
 \diver\left(\frac{1}{b}\nabla\psi^0(t,\cdot)\right)=b\omega(t,\cdot) \qquad \text{in} \quad \mathcal{D}'(\Oo).
\end{equation*}
Then, we argue as in the proof of Lemma~\ref{lem:convergence-phie}: as $C^\infty_{c}(\Weps)$ is dense in $X_{\varepsilon}$ we know that there exists a sequence of smooth compactly supported $\varphi_{n}$ which tends to $\psi^0_{\varepsilon}$ in the $X_{\varepsilon}$ norm, which is also true in the $X_{b}$ norm because $\beps\leq Cb$. Therefore, we have for a.e. $t\in \R_+$
\[
-\int_{\Omega} \frac1b \nabla\psi^0(t,x) \cdot \nabla \psi^0_{\varepsilon}(t,x) \dd x = \int_{\Omega} b\omega(t,x) \psi^0_{\varepsilon}(t,x)\dd x.
\]
Next, we notice that the sequence $\sqrt{b}^{-1} \nabla \psi^0_{\varepsilon}$ is bounded in $L^\infty(\R_{+};L^2(\Omega))$, hence it converges weakly-$\ast$ in $L^\infty(\R_{+};L^2(\Omega))$ to a $L^\infty L^2$ function which is a.e. equal to $\sqrt{b}^{-1} \nabla \psi^0(t,\cdot)$. This allows us to state that for a.e. $t\in \R_{+}$
\[
-\Big\| \frac1{\sqrt b} \nabla\psi^0(t,\cdot) \Big\|_{L^2(\Omega)}^2 = \int_{\Omega} b\omega(t,x) \psi^0(t,x)\dd x.
\]
For a.e. $t\in \R_{+}$, we exploit the density of $C_c^{\infty}$ in $X_{\eps}$ to state that \eqref{eq:psi0eps} also holds with $\psi_{\eps}^0$ as a test function, hence
\begin{multline*}
 \int_0^T\int_{\Weps}\frac{1}{\beps}\left|\nabla\psi_\eps^0\right|^2\dd x \dd t=-\int_0^T\int_{\Weps}\beps\weps\psi_\eps^0\dd x \dd t\\
 \rightarrow -\int_0^T\int_{\Omega}b\omega\psi^0 \dd x \dd t=\int_0^T\int_{\Omega}\frac{1}{b}\left|\nabla\psi^0\right|^2 \dd x \dd t,
\end{multline*}
where we used the strong $L^2$-convergence of $\psi_{\eps}^0$. The desired strong $L^2$-convergence of $\sqrt{\beps}^{-1}\nabla\psi_\eps^0$ follows.
\end{proof}

\subsection{Proof of Theorem~\ref{thm:maingeneral}}
First, we prove strong convergence to the target system posed on $\R_+\times \Oo$.
\begin{proposition}\label{prop:puncturedlake}
Let $(\Weps,\beps)$ be a sequence of lakes with one non-degenerated island (as in Definition~\ref{def-generallake}), which converges (in the sense of Assumption~\ref{ass:evanescent}) to a punctured lake $(\Omega,b)$ (as in Definition~\ref{ass:b}). Given $\gamma\in \R$ and $\omega^0\in L^{\infty}(\Omega)$, let $(\veps,\weps)$ be a global weak solution given in Theorem~\ref{thm:existencenonsmooth} with initial vorticity $\omega^0$ and initial circulation $\gamma \in \R$. As $\eps\rightarrow 0$, there exists a subsequence still denoted by $(\veps,\weps)$ such that
 \begin{equation*}
\begin{aligned}
 \sqrt{\beps}\veps&\rightarrow \sqrt{b}v \quad \text{strongly in} \quad L^{2}_{\loc}(\R_{+};L^2(\Omega)),\\
 \weps&\rightharpoonup^{\ast} \omega \quad \text{weakly-$\ast$ in} \quad L^{\infty}(\R_{+}\times \Omega),
\end{aligned}
\end{equation*}
where $(v,\omega)$ is a solution of the vorticity formulation of the lake equations in $\R_{+}\times \Oo$ with initial data $(\omega^0,v^0)$ and $\gamma(v)=\gamma(v^0)=\gamma$ and satisfies the Hodge decomposition
\begin{equation}\label{eq:Hodgelimit}
 v=\frac{1}{b}\left(\nabla^{\perp}\psi^ 0+\alpha\nabla^{\perp}\psi^1\right), \quad \alpha(t)=\gamma+\int_{\Omega}b\omega(t,\cdot)\phi^1\dd x.
\end{equation}
with $\phi^1, \psi^1, \psi^0$ as in Lemmas~\ref{lem:convergence-phie}, \ref{lem:bharmonic} and \ref{lem:Dirichlet} respectively. In particular, $\curl(v)=b\omega$ in $\mathcal{D}'(\R_+\times\Oo)$.
\end{proposition}

\begin{proof}
Given a lake $(\Omega,b)$ satisfying Definition~\ref{ass:b}, let $(\veps,\weps)$ be a sequence of weak solutions to the lake equations posed on $(\Weps,\beps)$ as in Assumption~\ref{ass:evanescent} and provided by Theorem~\ref{thm:existencenonsmooth}. The weak convergence for $\omega_{\varepsilon}$ comes directly from \eqref{bd:vorticity}, which also gives the uniform estimate for $\omega$.

We recall that by means of the Hodge decomposition one has 
\begin{equation*}
 \sqrt{\beps}\veps=\frac{1}{\sqrt{\beps}}\nabla^{\perp}\psi_{\eps}^0+\frac{\alphae}{\sqrt{\beps}}\nabla^{\perp}\psi_\eps^1,
\end{equation*}
with 
\begin{equation*}
 \alphae(t)=\gamma+\int_{\Weps}\beps\weps(t,x)\phie(x)\dd x.
\end{equation*}
We proved in Lemma~\ref{lem:veps2} that $\alphae$ is uniformly bounded in $\R_{+}$. In addition, we have that $\gamma$ is conserved and for all $\varphi\in C^\infty_{c}((0,\infty))$
\begin{align*}
\langle \partial_t\alphae, \varphi\rangle&= - \int_{0}^\infty\int_{\Weps} \beps\weps \partial_{t} (\phie \varphi) \dd t\dd x \\
&=\int_{0}^\infty\int_{\Weps}\beps\veps\weps \cdot \nabla (\phie \varphi ) \dd t\dd x ,
\end{align*}
where we have used that \eqref{eq:vorticity} is also true for $\Phi \in C^\infty_{c}(\R_{+}\times \overline{\Weps})$ (see the second remark after Theorem~\ref{thm:maingeneral}), the density of $C^\infty_{c}(\Weps)$ in $X_{\varepsilon}$ and the decomposition $\phie=\tilde\phie+\chi_{\delta}$.
As $\sqrt{\beps}\veps\sqrt{\beps}\weps$ is uniformly bounded in $L^{\infty}(0,T;L^2(\Weps))$ as consequence of Assumption~\ref{ass:evanescent}, Theorem~\ref{thm:existencenonsmooth}, Lemma~\ref{lem:veps2}, and $\phie\in H^1(\Weps)$ uniformly bounded from Lemma~\ref{lem:uniformphieps}, we obtain $\alphae\in W^{1,\infty}(\R_+)$ uniformly bounded. Hence, there exists $\alpha\in W^{1,\infty}(\R_+)$ such that 
\begin{equation}\label{eq:convergencealpha}
 \alphae\rightharpoonup^{\ast} \alpha \quad \text{in} \quad W^{1,\infty}(\R_+), \qquad \alphae \rightarrow \alpha \quad \text{in}\quad L_{\loc}^{\infty}(\R_+),
\end{equation}
up to passing to a subsequence. Therefore, combining Lemma~\ref{lem:bharmonic}, Lemma~\ref{lem:Dirichlet} and \eqref{eq:convergencealpha}, we conclude that 
\begin{equation*}
 \sqrt{\beps}\veps\rightarrow \sqrt{b}v \quad \text{strongly in} \quad L_{\loc}^2(\R_+,L^2(\Omega)),
\end{equation*}
where $v$ satisfies the Hodge decomposition \eqref{eq:Hodgelimit}. We have identified $\alpha$ thanks to Lemma~\ref{lem:convergence-phie}.

Next, we pass to the limit in \eqref{eq:vorticity}. As $\Weps$ converges to $\Oo$ in Hausdorff sense, given $\Phi\in C_c^{\infty}(\R_+\times \Oo)$ one has $\supp(\Phi)\subset \R_+\times \Weps$ for all $\eps$ sufficiently small. Hence, \eqref{eq:vorticity} is verified for $\Phi\in C_c^{\infty}(\R_+\times \Oo)$ and $\eps$ sufficiently small. The strong convergence of $\sqrt{\beps}\veps$ in $L_{\loc}^2(\R_+\times \Omega)$ together with the weak convergence of $\beps\weps$ and $\sqrt{\beps}\weps$ from \eqref{eq:convergence bepsweps} is sufficient to pass to the limit in \eqref{eq:vorticity}. We notice that as $\sqrt{\beps}\veps^0\in L^2(\Weps)$ uniformly bounded, we can follow the same lines to conclude that $\sqrt{\beps}\veps^0\to \sqrt{b}v^0$ in $L^2(\Omega)$ where $v^0$ satisfies the Hodge decomposition \eqref{eq:Hodgelimit}. 

Further, if follows from \eqref{eq:Hodgelimit} that for a.e. $t\in \R_+$, we have $\curl(v(t,\cdot))=b\omega(t,\cdot)$ in $\mathcal{D}'(\Oo)$. Moreover, due to the previous convergences, one may pass to the limit in \eqref{imperm}-\eqref{imperm2} to get the impermeability and divergence free condition for $v^0$ and $v$. Finally, we observe that passing to the limit in \eqref{eq:gencirc} yields $\gamma(v(t,\cdot))=\gamma(v^0) = \gamma$ for a.e. $t\in \R_+$.
\end{proof}

\begin{remark}\label{rem:testsupport}
 As said in the forth remark after Theorem~\ref{thm:maingeneral}, $(\veps,\weps)$ constructed in Theorem~\ref{thm:existencenonsmooth} verifies \eqref{eq:vorticity} also for $\Phi \in C^\infty_{c}(\R_{+}\times \overline{\Weps})$. Hence, we can notice that the previous proof works fine for $\Phi\in C_c^{\infty}(\R_+\times \overline{\Omega})$ and we state that $(v,\omega)$ satisfies \eqref{eq:vorticity} for such test functions. 
 
In the previous proposition, we have stated that $(v,\omega)$ is a solution to the vorticity formulation in $\R_{+}\times \Oo$ for two reasons. First, we think that it is an interesting notion because the point is a non-erasable singularity in the weighted norm, which is the natural framework for the lake equation (see Lemma~\ref{lem:capacity}). Second, we need additional arguments to compute $\curl v$ in $\mathcal{D}'(\Omega)$.
\end{remark}

\begin{remark}\label{rem:vLp}
Up to now, we have never used that $a_{1}<2$. Hence, if we are only interested in the existence of global weak solution in $\Oo$, we may relax the assumptions on $a_1$, see also the related Remark~\ref{rem:genass} for possible generalizations. The assumption $a_1\in(0,2)$, namely (\emph{3}) of Definition~\ref{ass:b} together with Remark~\ref{rem:b} yield that $v\in L^{\infty}(\R_+;L_{\loc}^p(\Omega))$ with $\frac{1}{p}=\frac{1}{2}+\frac{1}{q}$. If further $a_0\in[0,1)$ then $\sqrt{b}^{-1}\in L^q(\Omega)$ and $v\in L^{\infty}(\R_+;L^p(\Omega))$. If $v\in L_{\loc}^1(\Omega)$, then we may identify $\curl(v)$ in $\mathcal{D}'(\R_+\times\Omega)$ which is crucial for recovering the asymptotic equation on $\R_+\times\Omega$.
\end{remark}

We need the following property of distributions supported in $\{0\}\subset \Omega$ to pass to the limit in \eqref{eq:vorticity} posed on $\R_+\times\Omega$.
\begin{lemma}\label{lem:delta}
Let $\Omega$ be an open simply connected set in $\R^2$ and $0\in \Omega$. Let $T\in \mathcal{D}'(\Omega)$ be such that $\supp(T)=\{0\}$. Then, there exists a multi-index $\alpha$ and real coefficients $a_{\alpha}$ such that 
\begin{equation*}
 T=\sum_{|\alpha|\leq k}a_{\alpha}\partial^{\alpha}\delta_0,
\end{equation*}
where $k\in \N$ is given by the order of $T$. If in addition $T\in W^{-1,p}(\Omega)$ with $p\in[1,2)$, then $T=a_0\delta_0$.
\end{lemma}

\begin{proof}
We remind that a distribution with compact support is of finite order and \cite[Theorem 2.3.4]{H} yields the first part of the statement. For the second part, we need to show that $T$ is of order $0$. As $C^1(\Omega)\subset W^{1,p'}(\Omega)\subset C^0(\Omega)$, if follows that the order $k$ of $T$ is at most $1$. Assume by contradiction that $k=1$, then there exists $\alpha_1\neq 0$ with $|\alpha_1|=1$ such that $T=\sum_{|\alpha|\leq 1}a_{\alpha}\partial^{\alpha}\delta_0$. It is then easy to verify that $T\notin W^{-1,p}(\Omega)$. Indeed, it suffices to consider $f=|x|^{1-\frac{2}{p'}+}\in W^{1,p'}(\Omega)$ for which the pairing $\left\langle (a_{1}\partial_{1}+a_{2}\partial_{2})\delta_0, f\right\rangle$ is not well-defined when $(a_{1},a_{2})\neq(0,0)$, contradicting $T\in W^{-1,p}(\Omega)$. Hence $k=0$ and $T=a_0\delta_0$.
\end{proof}

We are now in position to prove Theorem~\ref{thm:maingeneral}.
\begin{proof}[Proof of Theorem~\ref{thm:maingeneral}]
We have already noted in Remark~\ref{rem:testsupport} that \eqref{eq:vorticity} is satisfied in $\mathcal{D}'(\R_+\times \Omega)$. Hence, it suffices to show that $\curl(v)=b\omega+\gamma\delta_0$ and the theorem will then follow from Proposition~\ref{prop:puncturedlake}. 

As Remark~\ref{rem:vLp} yields that there exists $p\in[1,2)$ such that $v\in L^{\infty}(\R_+;L_{\loc}^p(\Omega))$ and that for a.e. $t\in \R_+$,
\begin{equation*}
 \curl v(t,\cdot)=b\omega(t,\cdot) \qquad \text{in} \quad \mathcal{D}'(\Oo),
\end{equation*}
we state that the distribution defined as $T:=\curl v(t,\cdot)-b\omega(t,\cdot)$ is supported on $\{0\}$ and belongs to $W_{\loc}^{-1,p}(\Omega)$. Hence, Lemma~\ref{lem:delta} yields that $T=\beta\delta_0$ for some $\beta\in \R$. Moreover, we have passed to the limit for the generalized circulation $\gamma(\veps)$ to prove that
\begin{equation*}
 \gamma=\gamma(v(t,\cdot))=-\int_{\Omega}\nabla^{\perp}\chid\cdot v(t,x)\dd x -\int_{\Omega}\chid b\omega(t,x)\dd x.
\end{equation*}
Thus, we conclude
\begin{equation*}\begin{split}
 \beta=\left\langle T,\chid\right\rangle&=\left\langle \curl v(t,\cdot)-b\omega(t,\cdot), \chid\right\rangle\\
&=-\int_{\Omega}\nabla^{\perp}\chid\cdot v(t,x)\dd x-\int_{\Omega}\chid b\omega(t,x)\dd x=\gamma
\end{split}
 \end{equation*}
which reads
\begin{equation*}
 \curl v(t,\cdot)=b\omega(t,\cdot)+\gamma\delta, \qquad \text{in} \quad \mathcal{D}'(\Omega),
\end{equation*}
for a.e. $t\in \R_{+}$, and then completes the proof.
\end{proof}

\begin{remark}\label{rm:univ0}
We finish this section by noticing that the compactness argument provides us the existence of a vector field $v^{0}$ verifying $\sqrt{b}v^0\in L^2(\Omega), $\eqref{imperm} and $\curl v^0= \omega^0+\gamma\delta_{0}$. As the punctured lake under consideration verifies the assumption for the density of $C^\infty_{c}(\Omega)$ in $X_{b}$, Corollary~\ref{coro:BSlimit} states that such $v^0$ is unique. As $v^0$ is exactly the same as in Proposition~\ref{prop:puncturedlake}, we deduce from this analysis, which passes by the formulation in $\Omega$, that the initial data $v^0$ is uniquely determined by the curl in $\Oo$ (namely $b\omega^0$) and the generalized circulation (namely $\gamma$), even if the island is degenerated (with $a_{1}<2$ in order to use Remark~\ref{rem:vLp}). 
\end{remark}

\section{Compactness for the emergent island}\label{sec:emergent}

This section is dedicated to the proof of Theorem~\ref{thm:main2general}. The proof is easier than for the evanescent island because the Hodge decomposition is simpler and we do not need to estimate harmonic functions. The important informations that we need is that $\sqrt b^{-1}$ belongs to $L^q_{\loc}(\Omega)$ for some $q>2$ and that $C^\infty_{c}(\Omega)$ is dense in $X_{b}(\Omega)$, both properties coming from Definition~\ref{ass:b}, see Appendix~\ref{app:limit}. However, there is a new difficulty when we do not assume $\Weps\subset \Omega$: we have to infer from Hausdorff convergence that the limit of functions in $H^1_{0}(\Weps)$ satisfies the Dirichlet boundary condition on $\partial\Omega$.

Let $D\subset \R^2$ be a simply connected smooth open domain such that $\Omega, \Weps\subset D$ for all $\eps>0$. We extend $\weps$ by zero to $D$. As before, there exists $\omega\in L^{\infty}(\R_+\times D)$ such that $\weps\rightharpoonup^{\ast} \omega$ in $L^{\infty}(\R_+\times D)$ up to passing to a subsequence. From the elementary identity
\begin{equation*}
 \beps-b=\sqrt{\beps}\sqrt{b}(\sqrt{\beps}+\sqrt{b})\left(\frac{1}{\sqrt{\beps}}-\frac{1}{\sqrt{b}}\right)
\end{equation*}
and the uniform $L^{\infty}$-bounds from Definition~\ref{ass:b} and Assumption~\ref{ass:emergent},
we infer that $\beps\rightarrow b$ strongly in $L_{\loc}^p(\Omega)$ for all $p\in [1,\infty)$. We also have that $\beps\weps$ converges weakly star to $b\omega$ in $L^{\infty}(\R_+;L_{\loc}^p(\Omega))$ for all $p\in [1,\infty)$.

Next, we show the analogous of Lemma~\ref{lem:Dirichlet} for the emergent island.

\begin{lemma}\label{lem:convergence-psie}
Let $(\Weps,\beps)$ be a sequence of lakes without island (as in Definition~\ref{def-generallake}), which converges (in the sense of Assumption~\ref{ass:emergent}) to a punctured lake $(\Omega,b)$ (as in Definition~\ref{ass:b}). Given $\omega^0\in L^{\infty}(\Omega)$, let $(\veps,\weps)$ be a global weak solution given in Theorem~\ref{thm:existencenonsmooth} with initial vorticity $\omega^0$. As $\eps\rightarrow 0$, there exists a subsequence still denoted by $(\veps,\weps)$ such that
\begin{equation*}
\sqrt{\beps}\veps \rightarrow \sqrt{b}v \qquad \text{in} \quad L_{\loc}^2(\R_+;L^2(\Omega)),
\end{equation*}
where $v$ is such that $\curl v=b\omega$ in distributional sense and 
\begin{equation*}
 \diver(bv)=0, \quad \text{in} \quad \Omega, \qquad (bv)\cdot \mb n=0 \quad \text{on} \quad \partial\Omega,
\end{equation*}
in weak sense, see \eqref{imperm2}. In particular, there exists $\psi\in L^{\infty}(\R_{+};X_b(\Omega))$ such that 
\begin{equation*}
 v=\frac{1}{b}\nabla^{\perp}\psi.
\end{equation*}
Moreover, $v\in L^{\infty}_{\loc}(\R_+;L_{\loc}^p(\Omega))$ with $\frac{1}{p}=\frac{1}{2}+\frac{1}{q}$, where $q$ as defined in Definition~\ref{ass:b}.
If further $a_0\in[0,1)$, then $v\in L^{\infty}(\R_{+};L^p(\Omega))$.
\end{lemma}

\begin{proof}
We recall from Lemma~\ref{lem:elliptic} and Corollary~\ref{lem:veps1} that $\sqrt{\beps}\veps\in L^{\infty}(\R_+;L^2(\Weps))$ uniformly bounded and $\veps=\beps^{-1}\nabla^{\perp}\psi_{\eps}$ with $\psi_{\eps}\in L^{\infty}(\R_{+};X_{\eps})$ unique solution to \eqref{eq:elliptic} with $f_{\varepsilon}=\beps\weps$. We reproduce the argument in the proof of Lemma~\ref{lem:Dirichlet} to state that 
$\psi_\eps^0\in W^{1,\infty}(\R_+;X_{\eps})$ uniformly bounded. Upon extending by zero, by Poincar\'e inequality in $D$ we have $\psi_{\eps}^0\in W^{1,\infty}(\R_+;H_0^1(D))$ uniformly bounded. 
Thus, there exists $\psi\in W^{1,\infty}(\R_+;H^1_{0}(D))$ such that up to extracting subsequences,
\begin{equation*}
 \psi_{\eps}^0\rightharpoonup \psi \quad \text{in} \quad L^{\infty}_{\loc}(\R_{+};H_0^1(D)), \qquad \psi_{\eps}^0\rightarrow \psi \quad \text{in} \quad L^{\infty}_{\loc}(\R_{+};L^2(D)).
\end{equation*}
As the sequence of domains $\Weps$ converges in Hausdorff sense to $\Omega$, it $\gamma$-converges to $\Omega$ and therefore $\psi\in W^{1,\infty}_{\loc}(\R_{+};H_0^1(\Omega))$, see \cite[App. C]{GV-Lacave1} or \cite[App. B]{LNP}.

Moreover, there exists $f\in W^{1,\infty}(\R_+;L^2(\Omega))$ such that 
\begin{equation*}
 \frac{1}{\sqrt{\beps}}\nabla\psi_\eps^0\rightharpoonup^{\ast} f \quad \text{in} \quad W^{1,\infty}(\R_+;L^2(D)),
\end{equation*}
where we can identify $f=\frac{1}{\sqrt{b}}\nabla\psi^0$ a.e. in $\R_+\times\Omega$ proceeding as in the proof of Lemma~\ref{lem:convergence-phie}.
Therefore, $\psi^0\in W^{1,\infty}(\R_+;X_b(\Omega))$. 

Next, we notice that 
\begin{equation*}
\left\langle\diver\left(\frac{1}{\beps}\nabla\psi_\eps\right), \varphi\right\rangle=\left\langle \beps\weps, \varphi\right\rangle
\end{equation*}
for all $\varphi\in C_c^{\infty}(\R_+\times \Omega)$ and $\eps$ sufficiently small as consequence of \eqref{eq:elliptic0} and (\emph{1}) Assumption~\ref{ass:emergent}. Exploiting (\emph{2}) Assumption~\ref{ass:emergent} and the weak $L^2$-convergence of $\sqrt{\beps}^{-1}\nabla\psi_{\eps}$ to $\sqrt{b}^{-1}\nabla\psi$ and the weak-$\ast$ convergence of $b_{\varepsilon}\omega_{\varepsilon}$ we pass to the limit as $\eps\to0$ yielding
\begin{equation*}
\left\langle\diver\left(\frac{1}{b}\nabla\psi\right), \varphi\right\rangle=\left\langle b\omega, \varphi\right\rangle
\end{equation*}
in $\mathcal{D}'(\R_+\times \Omega)$. Defining $v:=b^{-1}\nabla^{\perp}\psi$, we recover that $\curl(v)=b\omega$ in distributional sense. By virtue of the density of $C_c^{\infty}(\Omega)$ in $X_{b}(\Omega)$, one obtains for any $T>0$ that
\begin{equation*}
\int_{0}^T\int_{\Omega}\frac{1}{b}|\nabla\psi|^2\dd x \dd t=-\int_{0}^T\int_{\Omega}b\omega\psi \dd x \dd t.
\end{equation*}
Finally,
\begin{multline*}
\int_{0}^T\int_{\Weps}\frac{1}{\beps}\left|\nabla\psi_\eps\right|^2\dd x \dd t=-\int_0^T\int_{\Omega}\beps\weps\psi_{\eps}\dd x \dd t\\
\rightarrow -\int_{0}^T\int_{\Omega}b\omega\psi \dd x \dd t=\int_{0}^T\int_{\Omega}\frac{1}{b}|\nabla\psi|^2\dd x \dd t.
\end{multline*}
 which implies that $\sqrt{\beps}\veps$ converges strongly to $\sqrt{b}v$ in $L_{\loc}^2(\R_+\times D)$. 
 
Moreover, due to $bv =\nabla^\perp \psi$ where $\psi \in H_0^1(\Omega)$, one may prove that \eqref{imperm} holds true with $v(t,\cdot)$ for a.e. $t\in \R_{+}$, up to approximate $H_0^1(\Omega)$ functions by functions in $C^\infty_{c}(\Omega)$. We can also prove that the initial data $v^0$ satisfies the Hodge decomposition and \eqref{imperm}.
 
Definition~\ref{ass:b} and H\"older's inequality yield that $v\in L^{\infty}(\R_+;L_{\loc}^p(\Omega))$ with $1/p=1/2+1/q$. If $\sqrt{b}^{-1}\in L^q(\Omega)$, i.e. $a_0 <1$, then $v\in L^{\infty}(\R_+;L^p(\Omega))$.
\end{proof}

We are now in position to prove Theorem~\ref{thm:main2general} for the emergent island.
\begin{proof}[Proof of Theorem~\ref{thm:main2general}]
We notice that Lemma~\ref{lem:convergence-psie} states that the Hodge decomposition $v=\frac{1}{b}\nabla^{\perp}\psi$ with $\psi\in L^{\infty}(0,T;X_b(\Omega))$ holds, $\curl v=b\omega$ in distributional sense and
\begin{equation*}
\diver(bv)=0 \quad \text{in} \quad \Omega, \qquad (bv)\cdot\mb n=0\quad \text{on} \quad \partial \Omega,
\end{equation*}
in weak sense \eqref{imperm2}. It remains to pass to the limit in \eqref{eq:vorticity}. To that end, we recall that (\emph{i}) in Definition~\ref{ass:b} yields that $\Weps$ converges to $\Omega$ in Hausdorff sense. Hence, any $\varphi\in C_c^{\infty}(\R_+\times{\Omega})$ satisfies $\supp(\varphi)\subset \R_+\times\Weps$ for $\eps$ sufficiently small. Exploiting the strong $L_{\loc}^2$-convergence of $\sqrt{\beps}\veps$, the weak-$\ast$ $L^{\infty}$-convergence of $\weps$ and the weak $L_{\loc}^{p}$-convergence of $\beps\weps$, we pass to the limit in \eqref{eq:vorticity}. This ends the proof of Theorem~\ref{thm:main2general}.
\end{proof}

\begin{remark}
 In the previous proof, it is not obvious to show that \eqref{eq:vorticity} is valid for test functions supported up to the boundary $\varphi\in C_c^{\infty}(\R_+\times\overline{\Omega})$ because it is not clear how to extend them as a test function in $C_c^{\infty}(\R_+\times\overline{\Weps})$. This is easy if $\Weps\subset \Omega$ for all $\eps$ sufficiently small. 
 
 Nevertheless, such an extension is not very interesting because we have already obtained the global existence of the limit system for such test functions in Theorem~\ref{thm:maingeneral}. Moreover, it was proved that in \cite[Prop. A.5]{LNP}, that, if the solution is regular enough, it suffices to verify the equation for test functions in $C_c^{\infty}(\R_+\times \Omega)$ to prove that the equation is also true for test functions in $C_c^{\infty}(\R_+\times\overline{\Omega})$ and that the solution is unique. Of course, the main challenge consists in proving suitable regularity properties, for instance by adapting the Calder\'on-Zygmund inequalities for lakes with degenerated islands.
\end{remark}

\section{Velocity formulation}\label{sec:velocity}

This section discusses the $\eps$-limit for the velocity formulation of \eqref{eq:lake-vel}. To recover the asymptotic equation satisfied by the limit velocity field $v$ on the limit lake, one may consider two strategies,
\begin{itemize}
\item pass to the limit in the velocity formulation;
\item pass to the limit in the vorticity formulation, see Theorem~\ref{thm:maingeneral} and Theorem~\ref{thm:main2general} respectively, then recover the velocity formulation from the vorticity formulation.
\end{itemize}
Due to the degeneracy of $b$ in $0$, both strategies face major mathematical difficulties when considered on $\Omega$. Therefore, we first provide a rigorous result for the limit equation for $v$ posed on $\Oo$, namely with the lake reduced to the support of $b$, i.e. where the depth of the lake is positive, and second we discuss possible strategies leading to a velocity formulation on $\Omega$.

We limit our consideration to the case of a collapsing island including the creation of a point vortex, the case of the emergent island can be treated similarly with minor modifications. 

\subsection{Velocity formulation on the support of the depth function}
On the lake $(\Oo, b)$, we introduce the following notion of weak solutions. 
\begin{definition}\label{defi:velocity}
Given the lake $(\Omega, b)$ of Definition~\ref{ass:b}, let $v^0$ be a vector field such that $\sqrt{b}v^0\in L^2(\Omega)$ 
\begin{equation*}
 \diver(bv^0)=0 \quad \text{in} \quad \Omega, \qquad bv^0\cdot\mb n=0 \quad \text{on} \quad \partial\Omega,
\end{equation*}
in weak sense, see \eqref{imperm}, and $\curl(v^0)\in \mathcal{D}'(\Oo)$ with $b^{-1}\curl(v^0)\in L_{\loc}^{\infty}(\Oo)$. A velocity field $v$ is called a global weak solution to the velocity formulation of \eqref{eq:lake-vel} in $(\Oo,b)$ if 
\begin{enumerate}
 \item $\curl(v)\in\mathcal{D}'(\R_+\times\Oo)$ with $b^{-1}\curl(v)\in L_{\loc}^{\infty}(\R^{+}\times \Oo)$ and $\sqrt{b}v\in L^{\infty}(\R_+;L^2(\Omega))$;
 \item $\diver(bv)=0$ in $\Omega$ and $bv\cdot\mb n=0$ in $\partial\Omega$ in weak sense;
 \item the velocity formulation of \eqref{eq:lake-vel} is verified in distributional sense, namely for all $\Phi\in C_c^{\infty}([0,\infty)\times\Oo)$ such that $\diver(\Phi)=0$ it holds that
 \begin{equation}\label{eq:velocity}
 \int_0^{\infty}\int_{\Omega}v\cdot\partial_t\Phi+\left(bv\otimes v\right):\nabla\left(\frac{\Phi}{b}\right) \dd x \dd t+\int_{\Omega}v^0\Phi(0)\dd x=0.
 \end{equation}
\end{enumerate}
\end{definition}
For the scenario of the evanescent island, we obtain the following.
\begin{theorem}\label{thm:velocity}
Let $(\Weps,\beps)$ be a sequence of lakes with one non-degenerated island (as in Definition~\ref{def-generallake}), which converges (in the sense of Assumption~\ref{ass:evanescent}) to a punctured lake $(\Omega,b)$ (as in Definition~\ref{ass:b}). Further, we assume $\beps\in W_{\loc}^{1,\infty}(\Omega_{\varepsilon})$ and $\beps\to b$ in $W_{\loc}^{1,\infty}(\Oo)$. Let $(\veps,\weps)$ be a sequence of global weak solutions to \eqref{eq:velocity} on $(\Weps,\beps)$ with initial vorticity $\beps^{-1}\curl v^0_{\varepsilon}=\omega^0$ and circulation $\gamma$.

As $\eps$ goes to $0$, there exists a subsequence $\sqrt{\beps}\veps$ which converges strongly to $\sqrt{b}v$ in $L_{\loc}^2(\R_+;L^2(\Omega))$ where
\begin{equation*}
 v=\frac{1}{b}\nabla^{\perp}\psi^0+\frac{1}{b}\alpha\nabla^{\perp}\psi^1 \quad \text{a.e. in} \, \, \R_{+}\times \Omega 
\end{equation*}
is a global weak solution to \eqref{eq:velocity} in $(\Oo,b)$ with initial data $v^0$ and in particular $\curl(v)=b\omega$ in $\mathcal{D}'(\R_+\times \Oo)$. 
\end{theorem}

Existence of a global weak solution to the velocity formulation on $(\Weps,\beps)$ with initial data $(\weps^0,\veps^0)$ under consideration follows from \cite[Theorem 1.6]{LNP}.

\begin{proof}
We notice that for both $\eps>0$ and the limit lake, the velocity $\veps$ and $v$ respectively are uniquely determined by the vorticity $\weps$ and $\omega$ respectively and the circulation $\gamma$, see Proposition~\ref{prop:BS} and Remark~\ref{rm:univ0}. Moreover, it follows from \cite[Theorem 1.6]{LNP} that $(\veps,\weps)$ is a weak solution of \eqref{eq:vorticity}. Hence, Theorem~\ref{thm:maingeneral} yields that $\sqrt{\beps}\veps$ converges strongly to $\sqrt{b}v$ in $L_{\loc}^2(\R_+;L^2(\Omega))$ where $v$ is given by \eqref{eq:Hodgelimit} and $bv$ satisfies the divergence free condition \eqref{imperm} in weak sense. Next, we pass to the limit in \eqref{eq:velocity}.
The Hausdorff convergence of $\Weps$ to $\Oo$ implies that given $\Phi\in C_c^{\infty}(\R_+\times\Oo)$, one has $\supp(\Phi)\subset \R_+\times\Weps$ for all $\eps$ sufficiently small. Therefore, \eqref{eq:velocity} is satisfied in particular for any $\Phi\in C_c^{\infty}(\R_+\times\Oo)$ and $\eps$ sufficiently small. In particular, $\diver\Phi=0$ on $\Weps$. Next, $\veps$ converges strongly to $v$ in $L^2_{\loc}(\R_{+};L_{\loc}^p(\Oo))$ for all $p\in [1,2)$ as 
\begin{equation*}
\veps=\frac{1}{\sqrt{\beps}}\sqrt{\beps}\veps
\end{equation*}
where $\sqrt{\beps}\veps$ converges strongly in $L_{\loc}^{2}(\R_+;L^2(\Omega))$ while $\sqrt{\beps}^{-1}$ converges strongly to $\sqrt{b}^{-1}$ in $L_{\loc}^{q}(\Oo)$. For any divergence free $\Phi\in C_c^{\infty}(\R_+\times\Oo)$ we conclude that
\begin{equation*}
\int_{0}^{\infty}\int_{\Omega}\veps\partial_t\Phi\dd x \dd t\rightarrow \int_{0}^{\infty}\int_{\Omega}v \partial_t\Phi\dd x \dd t
\end{equation*}
as $\eps\rightarrow 0$. For the term involving the initial data, it suffices to notice that $\sqrt{\beps}\veps^0\in L^2(\Weps)$ converges strongly in $L^2_{\loc}(\Oo)$, hence $\veps^0=\sqrt{\beps}^{-1}\sqrt{\beps}\veps^0$ converges strongly in $L_{\loc}^p(\Oo )$ for $p\in [1,2)$. It remains to pass to the limit in the convective term
\begin{equation}\label{eq:nonlinear}
\int_0^{\infty}\int_{\Weps}\left(\sqrt{\beps}\veps\otimes\sqrt{\beps}\veps\right):\left(\frac{\nabla\Phi}{\beps}-\frac{\Phi}{\beps}\otimes\frac{\nabla\beps}{\beps}\right)\dd x \dd t
\end{equation}
which is possible by the $L^2$ convergence of $\sqrt{\beps}\veps$ and that
\begin{equation*}
\frac{\nabla\Phi}{\beps}-\frac{\Phi}{\beps}\otimes\frac{\nabla\beps}{\beps}\to \frac{\nabla\Phi}{b}-\frac{\Phi}{b}\otimes\frac{\nabla b}{b} \quad \text{in} \quad L^{\infty}_{\loc}(\Oo).
\end{equation*}
The proof is complete.
 \end{proof}

We notice that Theorem~\ref{thm:velocity} may also be proven by adapting the second strategy mentioned, namely to recover the velocity formulation from the vorticity formulation on the lake $(\Oo, b)$. 

\subsection{Formal velocity formulation on the limit lake}
The nonlinear term \eqref{eq:nonlinear} is in general not well-defined when considered on the limit lake $(\Omega,b)$ and for divergence-free test-functions $\Phi\in C_c^{\infty}(\R_+\times \Omega)$. Even, when only the circulation-free part of the velocity field is considered, this difficulty persists. For the respective problem for the $2D$ Euler equations, the asymptotic equation for the circulation-free part of the velocity field for the $2D$ Euler equations was derived in \cite{ILL1} exploiting regularity properties stemming from the explicit Biot-Savart law. For the lake equations, no explicit general representation of the kernel is known (see \cite{DekeyserVanS} for a result in this direction) and Calder\'on-Zygmund type inequalities are only proved for smooth lakes \cite{BM,LNP}. We recall that $v$ satisfies the Hodge decomposition \eqref{eq:Hodgelimit}. Let
\begin{equation*}
 \vreg:=\frac{1}{b}\nabla^{\perp}\psi^0+(\alpha-\gamma)\frac{1}{b}\nabla^{\perp}\psi^1, \qquad \vsing:=\gamma \frac{1}{b}\nabla^{\perp}\psi^1,
\end{equation*}
such that $v=\vreg+\vsing$. It follows from Theorem~\ref{thm:maingeneral} that $\vreg$ is circulation-free, more precisely 
\begin{equation}\label{eq:vreg}
 \diver(b\vreg)=0, \quad b\vreg\cdot\mb n=0, \quad \curl(\vreg)=b\omega \text{ in }\mathcal{D}'(\R_{+}\times \Omega).
\end{equation}
Therefore,
\begin{equation*}
 \curl(b\vreg)=2\nabla^{\perp}\sqrt{b}\cdot\sqrt{b}\vreg+b^2\omega\in L^p(\Omega),
\end{equation*}
with $1/p=1/2+1/q$, where $q$ as in Assumption~\ref{ass:evanescent}. As $\diver(b\vreg)=0$ and $\sqrt{b}\vreg \in L^{\infty}(\R_+;L^2(\Omega))$, the standard Calder\'on-Zygmund in $\Omega$ (assuming $\partial \Omega\in C^{1,1}$) allows us to state $b \vreg\in L^{\infty}(\R_+;W^{1,p}(\Omega))$. In particular, there exists $p\in(1,2)$ such that 
\begin{equation*}
 b\nabla\vreg\in L^{\infty}(\R_+;L^p(\Omega)), \quad \vreg\in L^{\infty}(\R_+;W_{\loc}^{1,p}(\Oo)).
\end{equation*}
On the other hand $\vsing$ is such that $\sqrt{b}\vsing\in L^2(\Omega)$ and
\begin{equation}\label{eq:vsing}
 \diver(b\vsing)=0, \quad b\vsing\cdot\mb n=0, \quad \curl(\vsing)=\gamma\delta_0 \text{ in }\mathcal{D}'(\R_{+}\times \Omega).
\end{equation}
These regularity properties are insufficient in order to recover the evolution equations for $\vreg$ or $v$. However, we provide a formal computation to obtain an asymptotic equation for $\vreg$ that is inspired by \cite[Section 5]{ILL1}. We observe that for $u,w$ smooth vector fields such that $\diver(bu)=\diver(bw)=0$ one has
\begin{equation*}
 \diver(bu\otimes w)=b(\nabla u)w=2b(\nabla u)^{asym}w+b(\nabla u)^Tw=bw^{\perp}\curl(u)+b(\nabla u)^Tw.
\end{equation*}
We compute,
\begin{align*}
 \diver(bu\otimes w+bw\otimes u)&=bw^{\perp}\curl(u)+bu^{\perp}\curl(w)+b(\nabla u)^Tw+b(\nabla w)^Tu\\
 &=bw^{\perp}\curl(u)+bu^{\perp}\curl(w)+b\nabla(u\cdot w).
\end{align*}
In particular, it follows from \eqref{eq:vreg} that
\begin{equation*}
 \diver(b\vreg\otimes \vreg)=(b\vreg)^{\perp}b\omega+\frac{b}{2}\nabla|\vreg|^2,
\end{equation*}
hence
\[
\curl\Big(\frac1b \diver(b\vreg\otimes \vreg)\Big) =\diver( \vreg b\omega).
\]
Note that the second contribution lacks a rigorous definition under the regularity properties considered here.
Taking into account \eqref{eq:vsing}, we formally infer
\begin{align*}
 \diver(b\vreg\otimes \vsing+&b\vsing\otimes \vreg)\\
 &=(b\vsing)^{\perp}\curl(\vreg)+(b\vreg)^{\perp}\curl(\vsing)+b\nabla(\vreg\cdot\vsing)\\
 &=(b\vsing)^{\perp}b\omega+\gamma(b\vreg)^{\perp}\delta_0+b\nabla(\vreg\cdot\vsing).
\end{align*}
In general, $\gamma(b\vreg(0))^{\perp}\delta_0$ lacks to be well-defined. Formally, one concludes that
\begin{equation*}
 \curl\left(\frac{1}{b}\diver\left(b\vreg\otimes \vsing+b\vsing\otimes \vreg\right)-\gamma\vreg(0)^{\perp}\delta_0\right)=\diver(\vsing b\omega).
\end{equation*}
It is then straightforward to compute
\begin{align*}
 0=&\partial_t(b\omega)+\diver(\vreg b\omega)+\diver(\vsing b\omega)\\
 =&\curl\Big(\partial_t\vreg+\frac{1}{b}\diver(b\vreg\otimes\vreg+b\vreg\otimes\vsing+b\vsing\otimes\vreg)\\
 &\qquad -\gamma\vreg(0)^{\perp}\delta_0\Big),
\end{align*}
amounting to \eqref{eq:continuity} for $(v=\vreg+\vsing,\omega)$.
Finally, one has
\begin{equation*}
 \partial_t \vreg+\vreg\cdot \nabla \vreg+\frac{1}{b}\diver(b\vsing\otimes\vreg+b\vreg\otimes\vsing)+\nabla p= \gamma \vreg(0)^{\perp}\delta_0.
\end{equation*}
As $\vreg(0)$ is not well-defined unless additional regularity, as available for smooth lakes \cite{BM, LNP}, is proven, the previous computations are not rigorous. Note that $\vreg\in W_{\loc}^{1,p}(\Omega)$ with $p>2$ would be sufficient. 

An alternative approach consists in replacing the weak formulation \eqref{eq:velocity} by the weak formulation of \eqref{eq:lake-vel} for test-functions $\Phi\in C^{\infty}(\Omega)$ such that $\diver(b\Phi)=0$ and $b\Phi\cdot \mb n=0$.
More precisely,
\begin{equation}\label{eq:velocityb}
 \int_0^\infty\int_\Omega(bv)\cdot\partial_t\Phi+(bv\otimes v):\nabla\Phi+ p\diver(b\Phi) \dd x \dd t+\int_{\Omega}(bv^0)\cdot \Phi(0)\dd x=0.
\end{equation}
Note that the pressure term cancels out for the chosen class of test-functions.
The nonlinear term in \eqref{eq:velocityb} is well-defined for $v$ of finite energy, in particular for $v$ as in Theorem~\ref{thm:maingeneral}. Nevertheless, \eqref{eq:velocityb} comes with two major flaws for our asymptotic analysis. First, reconstructing \eqref{eq:velocityb} from the vorticity formulation \eqref{eq:vorticity} seems to be difficult in the context of low regularity. Second, if one aims to perform the $\eps$-limit by writing \eqref{eq:velocityb} for $\eps>0$ and passing to the limit, one notes that the admissibility of the test-function is highly sensitive to both the geometry $\Weps$ and the depth $\beps$ through the incompressibility condition $\diver(\beps\Phi)=0$. Thus, one is led to consider approximate weak solutions for $\eps>0$ with several error terms. These appear in particular for the pressure term 
\begin{equation*}
 \int_0^\infty\int_{\Omega}p_{\eps}\diver(\beps\Phi)\dd x\dd t
\end{equation*}
that requires an accurate control. Due to the lack of uniform ellipticity for the equation
\begin{equation*}
 \diver\diver(\beps\veps\otimes\veps)=\diver(\beps\nabla p_\eps),
\end{equation*}
uniform estimates for the pressure are hard to obtain. While this strategy does not suit well for the stability problem considered here, in \cite{AT} the author relies on the viscous version of \eqref{eq:velocityb} to prove existence of weak solutions to the viscous lake equations posed on a simply connected lake $(\Omega,b)$ with $b$ power law type Muckenhoupt weight.

To conclude, it remains therefore an interesting and challenging question to derive the asymptotic velocity formulation on the lake $(\Omega,b)$ and test-functions whose support includes $0$.

\appendix

\section{Density of smooth functions}\label{app:limit}

The aim of this appendix is to show that $C_c^{\infty}(\Omega)$ and $C_c^{\infty}(\Weps)$ are dense in $X_b(\Omega)$ and $X_{\beps}(\Weps)$ respectively. It was shown in \cite{LNP} that $C_c^{\infty}(\Weps)$ is dense in $X_{\beps}(\Weps)$ provided that $(\Weps,\beps)$ is smooth. Firstly, we extend the proof of \cite{LNP} to lakes without degenerated island as in Definition~\ref{def-generallake}. Secondly, we infer the respective density property in $X_b$ for the punctured lake, as in Definition~\ref{ass:b}, for which we need to also deal with the degeneracy of $b$ in $0$.

As already pointed out in Section~\ref{sec:main} the density of $C_c^{\infty}(\Omega)$ in $X_{b}(\Omega)$ is important to ensure the existence and uniqueness of the Hodge decomposition. Even if the existence could also be derived by compactness, which is enough to prove Theorem~\ref{thm:maingeneral}, we deem that it is interesting to consider a domain where we have the uniqueness of the linear elliptic problem. Moreover, in Section~\ref{sec:emergent}, we use this density in the compactness argument, which allows to avoid the comparison between $X_{\varepsilon}$ and $X$ norms (see Remark~\ref{rem:genass}).

To encompass both type of lakes and to provide a generalization suitable for future works, we propose here a density result for the following lakes:
the lake $(\Omega,b)$ is a lake with $N_{nd}$ non degenerated islands $\partial\cal{I}^k$ ($N_{nd}\in \N$) and $N_{d}$ degenerated islands localized at $x_{p}$ ($N_{d}\in \N$)
\begin{enumerate}
\item $ \displaystyle \Omega := \widetilde{\Omega} \setminus \Bigl( \bigcup_{k=1}^{N_{nd}} \cal{I}^k \Bigl)$,
where $\widetilde \Omega$ is open bounded subset of $\R^2$ and $\cal{I}^k$ are disjoint compact simply connected subsets of $\widetilde \Omega$, such that $\partial\widetilde \Omega$, $\partial\cal{I}^k$ is a non trivial Jordan curve (i.e. not reduced to points);
\item $x_{p}\in \Omega$ for all $p\in \{1, \dots, N_{d}\}$;
\item $b\in L^\infty(\Omega, \R^+)$ such that, for any compact set $\displaystyle K\subset \Omega\setminus \bigcup_{p=1}^{N_{d}} \{x_{p}\} $, there exists positive number $\theta_K$ such that $ b(x) \ge \theta_K$ on $K$;
\item there are small neighborhoods $\mathcal{O}^0$, $\mathcal{O}^k$ and $\mathcal{O}_{p}$ of $\partial\widetilde{\Omega}$, $\partial\cal{I}^k$ and $\{x_{p}\}$ respectively, such that, for $k\in \{0, \dots, N_{nd}\}$ and $p\in \{1, \dots, N_{d}\}$,
\begin{equation*}
b(x)=c(x)\left[d(x)\right]^{a_k} \quad \text{ in } \mathcal{O}^k\cap \Omega, \qquad b(x)=c(x)|x-x_{p}|^{\tilde a_p} \quad \text{ in } \mathcal{O}_p\cap \Omega,
\end{equation*}
where $c(x)\ge \theta>0$ for all $x\in \Omega$, $d(x)=\dist(x,\partial\Omega)$, $a_{k}\geq 0$ and $\tilde a_{p} >0$;
\item if $b$ vanishes on the boundaries $\partial\widetilde \Omega$, $\partial\cal{I}^k$ (i.e. if $a_{k}>0$) then the respective boundary is a $C^1$ Jordan curve.
\end{enumerate}

Of course, the case $N_{d}=0$ or $N_{nd}=0$ is allowed. The aim of this appendix is to prove the following proposition.

\begin{proposition}\label{prop density} 
Under the assumptions listed above for the lake $(\Omega,b)$, the set 
\[
\Big\{ \varphi \in C^\infty_c(\Omega), \ \nabla \varphi =0 \text{ in }\bigcup_{p=1}^{N_{d}} B(x_{p},R) \text{ for some }R>0 \Big\}
\]
is dense in 
\[
 X_b(\Omega)=\Big\{\psi\in H_0^1(\Omega) \quad : \quad \frac{1}{\sqrt{b}}\nabla\psi\in L^2(\Omega)\Big\}
\]
w.r.t. the norm $\|\cdot \|_{X_b(\Omega)}= \|\frac{\nabla \cdot}{\sqrt{b}}\|_{L^2(\Omega)}$.
\end{proposition}

We divide this proof in three steps.

\noindent {\em Step 1: Hardy inequality close to boundaries with vanishing topography.}

In this first step, we derive an estimate in the neighborhood of the boundary where $a_{k}>0$
\begin{equation*}
\partial\Omega_R:=\Big\{x\in \bigcup_{k\in [0,N_{nd}]\ |\ a_{k}>0}\mathcal{O}_{k} :\ 0\le d(x)\le R\Big\},
\end{equation*}
where we prove an Hardy type inequality.

\begin{lemma}\label{HardyLem} There exists $R_{0},C>0$ such that the following inequality holds for every $f\in X_{b}(\Omega)$ and any positive $R\in (0,R_{0})$:
\begin{equation*}
\Vert b^{-1/2}(f/d)\Vert_{L^2(\partial\Omega_R)} \leq C \Vert b^{-1/2}\nabla f\Vert_{L^2(\partial\Omega_R)}.
\end{equation*}
\end{lemma}
\begin{proof}
 Such a lemma was proved in \cite{LNP} for smooth lake. We reproduce the proof here in order to show that the weakened assumptions on the lake are sufficient. 
 
We start with the following claim: there exists $R_{0},C>0$ such that for any $f\in H^1_0(\Omega)$, any positive $R\in (0,R_{0})$ and any $k\in[0,N_{nd}]$ such that $a_{k}>0$, there holds that
\begin{equation}\label{Hardy2}
\int_{\mathcal{O}^k \cap( \partial\Omega_{2R}\setminus \partial\Omega_R)} \vert f(x)\vert^2\, \dd x\leq C R^2 \int_{\mathcal{O}^k \cap\partial\Omega_{2R}}\vert\nabla f(x)\vert^2\, \dd x . 
\end{equation}
Indeed, the $C^1$ regularity of $\partial \cal I_{k}$ for $a_{k}>0$ (resp. $\partial \widetilde{\Omega}$ if $a_{0}>0$) allows us to apply the tubular neighborhood theorem to change variable in terms of the distance. In particular, for $R_{0}$ small enough, $d$ is a $C^1$ function.
Therefore, the claim follows directly from the fundamental theorem of Calculus and the standard H\"older's inequality at least for smooth compactly supported functions. By density, it extends to $H^1_0(\Omega)$.

Next, by \eqref{Hardy2}, we can write
\begin{equation*}
\begin{split}
\Big \Vert b^{-1/2} (f/d) &\Big \Vert_{L^2(\partial\Omega_R)}^2 = \sum_{n\in\mathbb{N}^*}\int_{\partial\Omega_{2^{1-n}R}\setminus \partial\Omega_{2^{-n}R}}\left(\frac{f(x)}{d(x)}\right)^2\frac{d x}{b(x)}\\
& \leq \frac1\theta \sum_{n\in\mathbb{N}^*} \sum_{k| a_{k}>0}(R2^{-n})^{-(a_k+2)}\int_{\cal O_{k}\cap(\partial\Omega_{2^{1-n}R}\setminus \partial\Omega_{2^{-n}R})} \vert f(x)\vert^2\, \dd x\\
&\leq C \sum_{n\in\mathbb{N}^*}\sum_{k | a_{k}>0} (R2^{-n})^{-a_k}\int_{\cal O_{k}\cap\partial\Omega_{2^{1-n}R}} \vert \nabla f(x)\vert^2\, \dd x\\
&\leq C \sum_{k| a_{k}>0} \int_{\cal O_{k}\cap \partial\Omega_R}\left(\sum_{n\in\mathbb{N}^*:\,\, d(x)\le 2^{1-n}R} (R2^{-n})^{-a_k}\right)\vert \nabla f(x)\vert^2\, \dd x.\\
\end{split}
\end{equation*}
Since the summation in the parentheses in the last line above is bounded by $(\frac d2)^{-a_{k}} /(2^{a_{k}}-1)$ hence by $Cb^{-1}$, the integral on the righthand side is bounded by $\Vert b^{-1/2}\nabla f\Vert_{L^2(\partial\Omega_R)}^2$. The lemma is thus proved. 
\end{proof}

\medskip
\noindent {\em Step 2: approximation in $X_{b}$ by fonctions which are constants close to the degenerated islands}

We cannot expect the same kind of estimate in the neighborhood of a degenerated island $\cal O_{p}$ because the standard Hardy inequality is critical in $L^2$ in dimension two. So we approximate any function in $X_b(\Omega)$ by functions in $X_b(\Omega)$ which are constant in the neighborhood of $x_{p}$.

\begin{lemma}\label{lem_const}
 For any $\varepsilon>0$ and $\psi\in X_b(\Omega)$, there exists $\varphi \in X_b(\Omega)$ and $R>0$ such that $\nabla \varphi =0$ on $\cup_{p=1}^{N_{d}} B(x_{p},R)$ and $\|\psi-\varphi\|_{X_{b}}\leq \varepsilon$.
\end{lemma}

\begin{proof}
Let $\varepsilon>0$ and $\psi\in X_b(\Omega)$ fixed, by the dominated convergence theorem, there exists $R_{\varepsilon}>0$ such that
\[
\Big\| \frac{\nabla\psi}{\sqrt b} \Big\|_{L^2(B(x_{p},R_{\varepsilon}))}\leq \varepsilon,
\]
for any $p$.

We introduce $\chi$ a smooth cutoff function such that $\chi(x)\equiv 1$ if $|x|\geq 1$ and $\chi(x)\equiv 0$ if $|x|\leq 1/2$. As the following computation is true for any $\psi \in C_{c}^\infty(\Omega)$
\begin{align*}
\int_{B(0,1)\setminus B(0,\frac12) }& \nabla\chi(x)\cdot (\nabla^\perp \psi)( R_{\varepsilon}x+x_{p})\, dx\\
&= \int_{B(0,1)\setminus B(0,\frac12) } \diver \Big[ \chi(x) (\nabla^\perp \psi)( R_{\varepsilon}x+x_{p}) \Big]\, dx\\
&= \int_{\partial B(0,1)} (\nabla^\perp \psi)( R_{\varepsilon}x+x_{p}) \cdot \mb n(x)\, d\sigma(x)\\
& = \int_{B(0,1) } \diver \Big[ (\nabla^\perp \psi)( R_{\varepsilon}x+x_{p})\Big]\, dx=0 ,
\end{align*}
we have by density that for any $\psi\in H^1_{0}(\Omega)$: 
\[
\int_{B(0,1)\setminus B(0,\frac12) } \nabla\chi(x)\cdot (\nabla^\perp \psi)( R_{\varepsilon}x+x_{p})\, dx = 0.
\]
It is known thanks to the Bogovski{\u\i} operator \cite{Bogovskii79,Bogovskii80} (see \cite[Theorem III.3.1]{Galdi}), that there exists $C>0$ depending only on the domain $A:=B(0,1)\setminus \overline{B(0,1/2)}$ such that the problem
\[
 \diver F_{p}(x) = \nabla\chi(x)\cdot (\nabla^\perp \psi)( R_{\varepsilon}x+x_{p}) , \quad F \in H^1_{0}(A)
\]
has a solution such that
\[
 \| F_{p} \|_{H^1(A)} \leq C \Big\| \nabla\chi(\cdot)\cdot (\nabla^\perp \psi)( R_{\varepsilon}\cdot +x_{p} ) \Big\|_{L^2 (A)}.
\]
Extending $F_{p}$ by zero in the exterior of $A$, we define
\[
\tilde F(x) :=\sum_{p}\chi\Big( \frac {x-x_{p}}{R_{\varepsilon}}\Big) \nabla\psi (x)+F^\perp_{p}\Big( \frac {x-x_{p}}{R_{\varepsilon}}\Big)
\]
where we verify that $\tilde F \equiv 0$ in $\cup_{p} B(x_{p},R_{\varepsilon}/2)$, 
\begin{align*}
\curl \tilde F (x)&= -\diver \tilde F^\perp\\
& = -\frac1{R_{\varepsilon}}\sum_{p}\Big[ (\nabla\chi)\Big( \frac {x-x_{p}}{R_{\varepsilon}}\Big) \cdot \nabla^\perp \psi( x) -(\diver F_{p})\Big( \frac {x-x_{p}}{R_{\varepsilon}}\Big) \Big]=0
\end{align*}
in $\cal{D}'(\Omega)$. Indeed, we have $\diver \chi \nabla^\perp \psi=\nabla \chi \cdot \nabla^\perp \psi$ when $\psi\in C^\infty_{c}(\Omega)$, hence by density we get the same equality in $\cal{D}'(\Omega)$ when $\psi\in H^1_{0}(\Omega)$.

Moreover,
\begin{align*}
\Big\|\frac{ \tilde F - \nabla \psi}{\sqrt b} \Big\|_{L^2(\Omega)} 
\leq& \Big\| \frac{ \nabla \psi}{\sqrt b} \Big\|_{L^2(\cup_{p}B(x_{p},R_{\varepsilon}))} \\
&+C\Big( \sum_{p} \frac{1}{R_{\varepsilon}^{a_{p}}}\Big\|F_{p}\Big( \frac {\cdot-x_{p}}{R_{\varepsilon}}\Big) \Big\|_{L^2(B(x_{p},R_{\varepsilon})\setminus B(x_{p},R_{\varepsilon}/2))}^2 \Big)^{1/2}\\
\leq& \varepsilon N_{d} + C\Big( \sum_{p} \frac{R_{\varepsilon}^2}{R_{\varepsilon}^{a_{p}}}\|F_{p}\|_{L^2(A)}^2 \Big)^{1/2}\\
 \leq &\varepsilon N_{d} + C\Big( \sum_{p} \frac{R_{\varepsilon}^2}{R_{\varepsilon}^{a_{p}}} \| \nabla\chi(\cdot)\cdot (\nabla^\perp \psi)( R_{\varepsilon}\cdot+x_{p}) \|_{L^2(A)}^2 \Big)^{1/2} \\
\leq &\varepsilon N_{d} + C\Big( \sum_{p} \frac{1}{R_{\varepsilon}^{a_{p}}} \| \nabla \psi \|_{L^2(B(x_{p},R_{\varepsilon})\setminus B(x_{p},R_{\varepsilon}/2))}^2 \Big)^{1/2}\\
 \leq &\varepsilon N_{d} +C \Big\| \frac{ \nabla \psi}{\sqrt b} \Big\|_{L^2(\cup_{p}B(x_{p},R_{\varepsilon}))}\leq C\varepsilon,
\end{align*}
where $C$ depends only on $\Omega$ and $b$.

As $\tilde F$ is curl free in $\Omega$ and $\oint_{\cal I_{k}} \tilde F\cdot\tau\, ds=\oint_{\cal I_{k}} \nabla \psi\cdot\tau\, ds =0$ for all $k\in [1,N_{nd}]$, there exists $\varphi\in L^2(\Omega)$ such that $\nabla \varphi=\tilde F$. Rigorously, the trace of $\nabla \psi$ may not be defined, nevertheless we can show that the generalized circulation of $\tilde F$ is well zero. Moreover, $\nabla (\varphi-\psi)=0$ in the connected set $\Omega\setminus \cup_{p}B(x_{p},R_{\varepsilon})$, so we can choose $\varphi=\psi$ in $\Omega\setminus \cup_{p}B(x_{p},R_{\varepsilon})$ which satisfies the same boundary condition, i.e
 $\varphi\in H^1_{0} (\Omega)$ verifies Lemma~\ref{lem_const}.
 \end{proof}

Thanks to the two previous lemmas, we can now adapt the proof of \cite{LNP}.

\medskip 

\noindent {\em Step 3: proof of Proposition~\ref{prop density}}
\begin{proof}
Fix $\varepsilon>0$ and $\psi \in X_b(\Omega)$. First, we use Lemma~\ref{lem_const} to introduce $\varphi\in X_b(\Omega)$ such that $\nabla \varphi \equiv 0$ on $\cup_{p}B(x_{p},R)$ for some $R>0$ and $\|\psi-\varphi\|_{X_{b}}\leq \varepsilon$.

Second, we construct a cut-off function $\chi\in C^1(\Omega)$ such that $\chi \equiv 1$ on $\cup_{p}B(x_{p},R) \cup_{a_{k}=0} \cal O_{k}$, $\chi \equiv 0$ in $\partial \Omega_{\tilde R}$ for some $\tilde R\in (0,R_{0})$ and
\begin{equation}\label{CutOff}
\Vert (1-\chi)\varphi\Vert_{X_{b}}\le\varepsilon.
\end{equation}
This $\chi$ is constructed in \cite{LNP} and is a consequence of Lemma~\ref{HardyLem}. We copy it here for convenience of the reader. Since $\varphi \in X_b(\Omega)$, there exists a positive $R_\epsilon$ such that
\begin{equation}\label{CutOff2}
\int_{\partial\Omega_{R_\epsilon}}\vert\nabla \varphi(x)\vert^2\frac{dx}{b(x)}\le\varepsilon^2.
\end{equation}
Let us introduce a cut-off function $\eta\in C^\infty(\mathbb{R}_+)$ such that $0\le\eta\le 1$, $\eta(z)\equiv1$ if $z\ge 1$ and $\eta(z)\equiv 0$ if $z\le 1/2$ and define
\begin{equation*}
\chi(x)=\eta(\dist(x, \partial \Omega_{0})/R_\epsilon),
\end{equation*}
where $\partial\Omega_{0}$ is the boundaries where the bottom vanishes $\cup_{k|a_{k}>0} \partial \cal I_{k}$ (with possibly $\partial\widetilde{\Omega}$).
Clearly, $\chi\in C^1(\Omega)$ thanks to the $C^1$ regularity of $\partial\Omega_{0}$ and verifies well the properties listed above \eqref{CutOff}. In addition, we note that $\nabla[(1-\chi)\varphi]=(1-\chi)\nabla \varphi-\varphi\nabla\chi$. It then follows by \eqref{CutOff2} that
\begin{equation*}
\int_{\Omega}(1-\chi(x))^2\vert\nabla \varphi(x)\vert^2\frac{dx}{b(x)}
\le \int_{\partial\Omega_{R_\epsilon}}\vert\nabla \varphi(x)\vert^2\frac{dx}{b(x)}
\le \varepsilon^2.
\end{equation*}
Meanwhile using the fact that
\begin{equation*}
\vert \varphi\nabla\chi\vert=\vert R_\epsilon^{-1} \varphi\eta^\prime(\dist(x, \partial \Omega_{0})/R_\epsilon) \nabla d(x) \vert\le |(\varphi/d)(x) \mathds{1}_{\partial \Omega_{R_{\varepsilon}}} | \Vert\eta^\prime\Vert_{L^\infty}
\end{equation*}
and Lemma~\ref{HardyLem}, we obtain
\begin{equation*}
\begin{split}
\int_{\Omega}\vert \varphi(x)\nabla\chi(x)\vert^2\frac{dx}{b(x)}&\le \Vert\eta^\prime\Vert_{L^\infty}\int_{\partial\Omega_{R_\epsilon}}\frac{\vert \varphi(x)\vert^2}{d(x)^2}\frac{dx}{b(x)}\leq C \int_{\partial\Omega_{R_\epsilon}} \vert\nabla \varphi(x)\vert^2\frac{dx}{b(x)}\\
&\leq C\varepsilon^2,
\end{split}
\end{equation*}
which ends the proof of \eqref{CutOff}.

Next, we split $\chi \varphi=\chi \tilde \chi \varphi + \chi (1-\tilde \chi)\varphi$ where $\tilde \chi\in C^1_{c}(\Omega)$ such that $\tilde \chi \equiv 1$ on $\cup_{p}B(x_{p},R)$. 

Second, we notice that $\chi (1-\tilde \chi)\varphi$ belongs to $H^1_{0}(\Omega)$ and $b(x)\geq \theta>0$ on his support. By definition of $H^1_{0}$ there exists $f\in C^\infty_{c}(\Omega)$, such that $\| f -\chi (1-\tilde \chi)\varphi \|_{X_{b}} \leq \theta^{-1/2} \| f -\chi (1-\tilde \chi)\varphi \|_{H^1}\leq \varepsilon$. Due to the support of $\chi (1-\tilde \chi)\varphi$ we can assume that $f\equiv 0$ on $\cup_{p}B(x_{p},R/2)$.

Third, we simply approximate the compactly supported function $\chi \tilde \chi \varphi$ with its $C^\infty_c$ mollifier functions. Indeed, if the mollifier parameter is small enough, $\chi \tilde \chi \varphi * \xi_{\delta}\equiv \varphi$ in $B(x_{p},R/2)$ so the support of $\nabla \Big(\chi \tilde \chi \varphi * \xi_{\delta}\Big)$ is at a distance $\eta>0$ from $\partial\Omega \cup_{p} \{ x_{p} \}$, so the convergence in $X_b(\Omega)$ norm comes from the convergence in $\dot H^1$ norm and Assumption (3) in the beginning of this appendix.

These three arguments give a $ C^\infty_c$ function $f+ \chi \tilde \chi \varphi * \xi_{\delta}$ which is at distance $C\varepsilon$ of $\psi$ in the $X_{b}$ norm and which is constant on $\cup_{p}B(x_{p},R/2)$. This completes the proof of Proposition~\ref{prop density}.
\end{proof}
The density property allows one to readily infer existence and uniqueness of solutions to the degenerated elliptic problems on $(\Omega,b)$ by reproducing the arguments of \cite[Section 2]{LNP}.

\begin{corollary}\label{coro:BSlimit}
Let $(\Omega,b)$ be a lake with one degenerated island as in Definition~\ref{ass:b}. Then, the Hodge decomposition provided by Proposition~\ref{prop:BS} holds true on $(\Omega,b)$.
\end{corollary}

Alternatively, we notice that we could have performed our analysis of Section~\ref{sec:evanescent} in a smaller space:
\[
\widetilde{X_{b}}(\Oo) = \text{ the closure for the $X_{b}$ norm of } C_c^\infty(\Omega\setminus\{0\}), 
\]
where the density is then encompassed in the definition. In general, the equivalence between $\widetilde{X_{b}}(\Oo)$ and $X_{b}(\Omega)$ does not hold and is a delicate question, see e.g. \cite{Z98}. On a related note, we remark that $\widetilde{X_b}(\Oo)\neq \widetilde{X_b}(\Omega)$ while $H_0^1(\Oo)=H_0^1(\Omega)$ which is related to the positive weighted $b^{-1}$-capacity of $\{0\}$ and the removability of the singularity in $\{0\}$. In particular, implementing the compactness argument in $\widetilde{X_b}(\Oo)$ would allow one to state that the stream-functions satisfy $\tilde\phi^1, \psi^0\in \widetilde{X_b}(\Oo)$, to be compared with Lemma~\ref{lem:convergence-phie} and \ref{lem:Dirichlet} respectively. However, we do not require this additional information for our asymptotic analysis. In particular, we can notice that in Section~\ref{sec:emergent}, $\psi^0\in \widetilde{X_b}(\Omega)$ and clearly not in $\widetilde{X_b}(\Oo)$.

Furthermore, as $\Cap_{b^{-1}}(\{0\})>0$, for any $\varphi\in X_b(\Omega)$ one may choose the quasi-everywhere continuous representative, i.e. $\varphi$ is continuous on sets of positive capacity and in particular in $0$.
We refer to the review paper \cite{K} and the monograph \cite{HKM}.

We finish this appendix by noticing that we do not need to assume $\tilde a_{p}<2$ to prove Proposition~\ref{prop density}. Of course, the obvious corollary of this proposition is that $C^\infty_{c}(\Omega)$ is dense in $X_{b}$, because $C^\infty_{c}(\Omega)$ is a larger set than the set appearing in this proposition. Nevertheless, we should be aware that $C^\infty_{c}(\Omega)$ is included in $X_{b}$ if and only if $\tilde a_{p}<2$ for every $p$. To show this equivalence, it suffices to consider test functions whose the gradient is constant in the neighborhood of $\{0\}$.

%%%%%%%%%%%%% Bibliography %%%%%%%%%%%%%%%%%%%%

\end{document}